\documentclass[10pt,a4paper]{elsarticle}                                                                            
\usepackage{amsmath,amsfonts,amsthm,amssymb}
\usepackage{natbib}

\usepackage{geometry}                        
\usepackage{graphicx}                          
\usepackage{setspace}  
\usepackage{caption}
\usepackage{listings}
\usepackage{tabularx}
\usepackage{xcolor}
\usepackage{indentfirst}
\usepackage{subcaption}

\usepackage{multirow}
\usepackage{overpic}
\usepackage{diagbox}

\usepackage{verbatim}

\usepackage{enumerate}

\usepackage[normalem]{ulem}

\usepackage{tikz}

\usepackage[title]{appendix}

\usepackage{bm}

\usepackage[font=footnotesize,labelfont=bf]{caption}

\usepackage[linesnumbered,ruled,lined,commentsnumbered]{algorithm2e}

\makeatletter
\@addtoreset{equation}{section}
\makeatother

\newtheorem{theorem}{Theorem}[section]

\newtheorem{remark}{Remark}[section]

\newcommand\dd{\mathrm{d}}
\newcommand\pp{\partial}

\newcommand\x{\bm{x}}
\newcommand\uvec{\mathbf{u}}
\newcommand\X{\mathbf{X}}
\newcommand\y{\bm{y}}

\newcommand\avec{\bm{\xi}}

\begin{document}

 \title{On Lagrangian schemes for porous medium type generalized diffusion equations: a discrete energetic variational approach}

\author[iit]{Chun Liu}
\ead{cliu124@iit.edu}
\author[iit]{Yiwei Wang}
\ead{ywang487@iit.edu}

\address[iit]{Department of Applied Mathematics, Illinois Institute of Technology, Chicago, IL 60616, USA}

\date{}
\begin{abstract}
  In this paper, we present a systematic framework to derive a Lagrangian scheme for porous medium type generalized diffusion equations by employing a discrete energetic variational approach. Such discrete energetic variational approaches are analogous to energetic variational approaches \cite{liu2009introduction, Giga2017} in a semidiscrete level, which provide a basis of deriving the ``semi-discrete equations'' and can be applied to a large class of partial differential equations with energy-dissipation laws and kinematic relations.
  The numerical schemes derived by this framework can inherit various properties from the continuous energy-dissipation law, such as conservation of mass and dissipation of the discrete energy.
As an illustration, we develop two numerical schemes for the multidimensional porous medium equations (PME), based on two different energy-dissipation laws.
We focus on the numerical scheme based on the energy-dissipation law with $\frac{1}{2} \int_{\Omega} |\uvec|^2 \dd \x$ as the  dissipation. Several numerical experiments demonstrate the accuracy of this scheme as well as its ability in capturing the free boundary and estimating the waiting time for the PME in both 1D and 2D.

%

\end{abstract}

\maketitle     

\section{Introduction}

Many mathematical models in physics, chemistry and biology can be viewed as generalized diffusions with different choices of free energy and dissipation functional from an energetic variational viewpoint \cite{Giga2017}. Examples include the porous medium equation (PME) \cite{vazquez2007porous}, Keller-Segel model \cite{keller1970initiation}, Poisson-Nernst-Planck (PNP) \cite{eisenberg2010energy, xu2014energetic} and Cahn-Hilliard equations \cite{cahn1958free, liu2019energetic}.

One simple example of a generalized diffusion equation is given by
\begin{equation}\label{PME}
  \begin{aligned}
    & \rho_t = \nabla \cdot (D(\rho, \x) \nabla \rho), \quad \x \in \Omega \subset \mathbb{R}^d,~ \alpha > 1,~ t > 0, \\
    & \rho(\x, 0) = \rho_0(\x), \quad \x \in \Omega, \\
    & \frac{\pp \rho}{\pp \mathbf{n}} = 0, \quad \x \in \pp \Omega,~ t > 0, \\
  \end{aligned}
\end{equation}
where $\rho$ is a non-negative function, $D(\rho, \x)$ is the diffusion coefficient, $\Omega$ is a bounded domain and $\mathbf{n}$ is the external normal direction in $\pp \Omega$. When $D(\rho, \x) = c \rho^{\alpha - 1}$ ($\alpha > 1$), (\ref{PME}) is the porous medium equation (PME)\cite{vazquez2007porous}.

There has been an increasing interesting in designing Lagrangian methods for diffusion equations \cite{carrillo2009numerical, westdickenberg2010variational, carrillo2016numerical, junge2017fully, carrillo2017numerical, matthes2017convergent, carrillo2018lagrangian}. The Lagrangian methods is particularly suitable for problems involving sharp interface and free boundary, especially for those with singularity. However, it is a common challenge to solve Lagrangian schemes numerically, especially in high spatial dimensions. Moreover, most of previous Lagrangian methods start with the nonlinear PDEs for the Lagrangian maps \cite{carrillo2016numerical, carrillo2017numerical, duan2019numerical}, it might be difficult to choose the proper weak form to preserve the original variational structure.

In this paper, we  present a systematic framework to construct Lagrangian schemes to generalized diffusion equations (\ref{PME}) by a discrete energetic variational approach, which can be easily applied to multiple spatial dimensions.  A discrete energy variational approach is an analogue to energetic variational approaches \cite{liu2009introduction, Giga2017} in a semidiscrete level, which provides a general framework to  derive the ``semi-discrete equations'', a system of ordinary differential equations in time, from a discrete energy-dissipation law. Our approach is quite close to the methods based on gradient flow structure in the $L_2$-Wasserstein metric \cite{carrillo2009numerical, westdickenberg2010variational, carrillo2016numerical, junge2017fully, carrillo2017numerical, matthes2017convergent, carrillo2018lagrangian}, but various energetic variational forms can be utilized in such a framework.   
As an illustration, we develop two numerical schemes for the multidimensional porous medium equation (PME), based on two different energy-dissipation laws.






The PME has a wide application in physical and biological problems, such as the flow of an ideal gas though a porous medium \cite{lacey1982waiting}, radiative transfer theory \cite{larsen1980asymptotic}, biological aggregation \cite{topaz2006nonlocal}, population dynamics \cite{witelski1997segregation},  
and tumor growth \cite{perthame2014hele}.
Theoretical studies \cite{oleinik1958cauchy, kalavsnikov1967formation, knerr1977porous, wilhelm1983quasilinear, aronson1983initially,  vazquez2007porous} have shown that the existence of the solutions to the PME, which will have a compact support at any time $t > 0$ if the initial data has a compact support. The interface between the compact support and zero-region, called the free boundary, will moves outward in a finite speed, known as the \emph{finite speed propagation}. Unlike the heat equation, the solution of the PME could become non-smooth even if the initial data is smooth. Moreover, for a certain initial data, the interface will not move until a finite positive time, called the \emph{waiting time} \cite{knerr1977porous, aronson1983initially}.

From a computational perspective, its lack of regularity and free boundary \cite{alikakos1985pointwise, aronson1983initially, vazquez2007porous} poses challenges for numerical simulations for the PME. For instance,  numerical solutions by standard numerical approaches, such as PCSFE (Predictor-Correction Algorithm and Standard Finite Element) method,  may contain oscillations near the free boundary, which cannot be removed by raising the degree of finite element space and/or by refining spatial meshes \cite{zhang2009numerical}. On the other hand, it is difficult to track the movement of free boundary and estimate the waiting time of the PME in high accuracy. During the past, quite a number of numerical methods have been proposed for the PME \cite{tomoeda1983numerical, gurtin1984coordinate, dibenedetto1984interface, hoff1985linearly, socolovsky1988lagrangian, socolovsky1988numerical, bertsch1990numerical, jager1991solution, monsaingeon2016explicit, jin1998diffusive, cavalli2007high, zhang2009numerical, liu2011high, budd1999self, baines2005moving, baines2006scale,  ngo2017study, ngo2018adaptive}, most of them are Eulerian methods and focus on the one-dimensional case.

A commonly used numerical approach is the interface tracking algorithm \cite{tomoeda1983numerical, gurtin1984coordinate, dibenedetto1984interface, hoff1985linearly, bertsch1990numerical, jager1991solution, monsaingeon2016explicit}, in which the interface equation are solved in Lagrangian coordinate, while the solution inside the numerical support is updated in Eulerian coordinate. However, it is not easy to apply such method to the PME in high spatial dimensions and the case with complex support. 
In order to eliminate the oscillations around the free boundary, some numerical methods from hyperbolic conservation law, such as relaxation scheme \cite{jin1998diffusive, cavalli2007high}, locally discontinuous Galerkin (LDG) \cite{zhang2009numerical} and WENO \cite{liu2011high}, have been applied to the PME successfully.
Since the solution of the PME always steeper near the interface, some adaptive moving mesh methods have proved to be useful in solving the PME \cite{budd1999self, baines2005moving, baines2006scale,  ngo2017study, ngo2018adaptive}, especially in multiple space dimensions. A typical moving mesh method updates the computational mesh according to the monitor function defined by the numerical solution at current time, which might effect the dynamic of evolution. Hence, it might be difficult to track the free boundary and estimate the waiting time in high accuracy within the adaptive moving mesh methods.

The rest of the paper is organized as follows. In section 2, we briefly review the energetic variational approaches for porous medium type generalized diffusions in a continuous level. In section 3, we introduce a discrete energetic variational approach and apply it to construct two numerical schemes for the PME based on two different energy-dissipation laws. The numerical experiments for the PME in one and two spatial dimensions are presented in sections 4.  We focus on the numerical scheme based on the energy-dissipation law with $\frac{1}{2} \int_{\Omega} |\uvec|^2 \dd \x$ as the  dissipation. These numerical results demonstrate the accuracy of our Lagrangian scheme as well as its ability in capturing the free boundary and estimating the waiting time for the PME in both 1D and 2D.


\section{Energetic Variational Approaches}

For a given energy-dissipation law and the kinematic (transport) relations, an energetic variational approach provides a general framework to determine the dynamics of system in a unique and well-defined way, through two distinct variational processes: Least Action Principle (LAP) and Maximum Dissipation Principle (MDP) \cite{liu2009introduction, Giga2017}. This approach is originated from  pioneering work of Onsager \cite{onsager1931reciprocal, onsager1931reciprocal2} and Rayleigh \cite{strutt1871some}, and has been successfully applied to build up many mathematical models for complex fluids \cite{liu2009introduction, sun2009energetic, eisenberg2010energy, Giga2017}.

The starting point of an energetic variational approach is the first and second laws of thermodynamics \cite{Giga2017}, which yields the following energy-dissipation law
\begin{equation}
\frac{\dd}{\dd t} E^{\text{total}}(t) = - 2 \mathcal{D}(t),
\end{equation}
for an isothermal closed system. $E^{\text{total}}$ is the total energy, which is the sum of the Helmholtz free energy $\mathcal{F}$ and the kinetic energy $\mathcal{K}$. $2\mathcal{D}$ is the rate of energy dissipation, which is related to the entropy production in thermodynamics. The Least Action Principle states that the dynamics of a conservative system is determined as a critical point of the action functional $\mathcal{A}(\x) = \int_0^T \mathcal{K} - \mathcal{F} \dd t$ with respect to $\x$ (the trajectory, if applicable) \cite{arnol2013mathematical, Giga2017}, i.e., 
\begin{equation}
  \delta \mathcal{A} =  \int_{0}^T \int_{\Omega_t} (f_{\text{inertial force}} - f_{\text{conservative force}})\cdot \delta \x  \dd \x \dd t,
\end{equation}
where $\Omega_t$ is the physical domain at time $t$. On the other hand, for a dissipative system ($\mathcal{D} \geq 0$), according to Onsager \cite{onsager1931reciprocal, onsager1931reciprocal2},  the dissipative force can be obtained by minimization of the dissipation functional $\mathcal{D}$ with respect to the ``rate'' $\x_t$, known as Maximum Dissipation Principle (MDP), i.e.,
\begin{equation}
\delta \mathcal{D} = \int_{\Omega_t} f_{\text{dissipative force}} \cdot \delta \x_t~ \dd \x.
\end{equation}
Then, according to the force balance (Newton's second law, in which the inertial force plays role of $ma$), we have
\begin{equation}\label{FB}
\frac{\delta A}{\delta \x} = \frac{\delta \mathcal{D}}{\delta \x_t}.
\end{equation}
We refer the reader to \cite{Giga2017} for more detailed description of energetic variational approaches. Here we only focus on the modeling of generalized diffusions by an energetic variational approach. 



Generalized diffusions are concerned with a conserved quantity $\rho(\x, t)$. An energetic variational approach to a diffusion based on a Lagrangian description to the system, in which the system is characterized by a flow map $\x(\X, t): \Omega_0 \rightarrow \Omega_t$ satisfies
\begin{equation*}
\x(\X, 0) = \X,
\end{equation*}
where $\X$ are the Lagrangian coordinates and $\x$ are Eulerian coordinates. For a given flow map $\x(\X, t)$, the velocity $\uvec(\x(\X, t), t)$ of the system is defined by
\begin{equation}
 \uvec(\x(\X, t), t) = \x_t(\X, t),
\end{equation}
and the deformation matrix (the deformation gradient) $F(\X, t)$ is the Jacobian matrix of the flow map $\x(\X, t)$:  
\begin{equation}
F(\X, t) = \nabla_{\X} \x(\X, t),
\end{equation}
which determines the kinematic relations of physical quantities in Lagrangian coordinates. 

For a conserved quantity $\rho(\x, t)$, let $\rho_0(\X)$ be the initial mass density, then the conservation of mass
\begin{equation}\label{con_mass}
\int_{\Omega_t} \rho(\x(t), t) \dd \x = \int_{\Omega_0} \rho_0(\X) \dd \X,
\end{equation}
indicates
\begin{equation}\label{rho_detF}
\rho(\x(\X, t), t) = \rho_0(\X) / \det F(\X, t).
\end{equation}

  In Eulerian coordinates, the conservation of mass (\ref{con_mass}) is equivalent to
  \begin{equation*}
    \begin{aligned}
      \frac{\dd}{\dd t} \int_{\Omega_t} \rho(\x(\X, t), t) \dd \x & =  \frac{\dd}{\dd t} \int_{\Omega_0} \rho(\x(\X, t), t) \det F(\X, t) \dd \X \\
      & = \int_{\Omega_t} \left( \rho_t + \nabla \rho \cdot \uvec + \rho (\nabla \cdot \uvec) \right) \dd \x = 0. \\
    \end{aligned}
  \end{equation*}
Hence, $\rho(\x(\X, t), t)$ satisfies a transport equation
  \begin{equation}\label{Transport}
    \rho_t + \nabla \cdot (\rho \uvec) = 0.
  \end{equation}

One can view (\ref{rho_detF}) as a kind of composition of the flow map $\x(\X, t)$ and the initial density $\rho_0(\X)$, i.e.,
\begin{equation}
  \rho(\x, t) = \rho_0 ~\tilde{\circ}~ \X^{-1}(\x, t) = \frac{\rho_0 ~\circ~ \X^{-1}(\x, t)}{\det F(\X, t)}, 
\end{equation}
where $\X^{-1}: \Omega_t \rightarrow \Omega_0$ is the inverse of the flow map $\x(\X, t)$.

\begin{remark}
  The scalar transport equation $\phi_t + (\uvec \cdot \nabla) \phi = 0$ is equivalent to $\phi(\x(\X, t), t) = \phi_0(\X)$ in the Lagrangian coordinates, which also can be viewed as a composition
  \begin{equation}
    \phi(\x, t) = \phi_0 \circ \X^{-1} (\x, t).
  \end{equation}
  The scalar transport equation is a suitable kinematic for designing a Lagrangian methods for the gradient flow system, such as Allen-Cahn equation. 
\end{remark}

\begin{remark}
The ``initial data'', or ``reference data'', $\rho_0(\X)$ ($\phi_0(\X)$) carries many information of the solutions, which may not be available for general problems. In practice, one may obtain $\rho_0(\X)$ ($\phi_0(\X)$) from other methods, such as those of Eulerian approaches. 
\end{remark}

In a porous medium type generalized diffusion considered in this paper, the overall energy-dissipation law of the system is given by
\begin{equation}\label{Diffusion}
\frac{\dd}{\dd t} \int_{\Omega} \omega(\rho) \dd \x = - \int_{\Omega} \eta(\rho) |\uvec|^2 \dd \x,
\end{equation}
where the kinetic energy $\mathcal{K}$ is neglected, $\omega(\rho)$ is the free energy density, and the energy dissipation $\mathcal{D}$ is taken to be $\frac{1}{2} \int_{\Omega} \eta(\rho) |\uvec|^2$ like the Darcy's law (the friction to the resting media) \cite{Giga2017}.

We first perform the LAP, i.e., compute the variation of $\mathcal{A} = \int_{0}^{T} - \mathcal{F} \dd t$ with respect to $\x(\X, t)$. Direct computation results in
\begin{equation}\label{LAP_Lag}
  \begin{aligned}
    & \delta \mathcal{A} = - \delta \int_{0}^T  \int_{\Omega_0} \omega(\rho_0(X)/ \det F) \det F \dd \X \\
    & = - \int_{0}^T \int_{\Omega_0} \left( - \omega_{\rho} \left( \frac{\rho_0(X)}{\det F} \right) \frac{\rho_0(X)}{\det F} + \omega\left(\frac{\rho_0(X)}{\det F}\right) \right) (F^{-\rm{T}} : \nabla_{\X} \delta \x)\det F  \dd \X, \\
      \end{aligned}
\end{equation}
where $\delta \x$ is the test function.
It can be noticed that even for this simple case, the variational result in Lagrangian coordinates is quite complicated, which involves $F^{-1}$ and $\det F$.

Pull (\ref{LAP_Lag}) back  to Eulerian coordinates and apply the integration by parts, one can get
\begin{equation}\label{LAP}
  \begin{aligned}
    \delta \mathcal{A} &  = - \int_{0}^T \int_{\Omega} \left( - \omega_{\rho} \rho + \omega \right) (\nabla_{\x} \cdot \delta \x)  \dd \x \dd t
     = - \int_{0}^T \int_{\Omega}  \nabla [ \omega_{\rho} \rho - \omega]  \cdot \delta \x  \dd \x \dd t,
    \end{aligned}
\end{equation}
where the boundary term vanishes due to the choice of $\delta \x$.

The MDP can be done by taking variation of $\mathcal{D}$ with respect to $\x_t$, one can easily obtain that 
\begin{equation}
f_{\text{dissipative force}}  = \frac{\delta \mathcal{D}}{\delta \x_t} =  \eta(\rho) \x_t.
\end{equation}

By the force balance ($\frac{\delta \mathcal{A}}{\delta \x} = \frac{\delta \mathcal{D}}{\delta \x_t}$), an energetic variational approach leads to  
\begin{equation}\label{Traj_PME}
\eta(\rho) \uvec = - \nabla p,
\end{equation}
where $p = \omega_{\rho} \rho - \omega$. The force balance equation (\ref{Traj_PME}) defines the velocity in the transport equation (\ref{Transport}). Combining (\ref{Transport}) and (\ref{Traj_PME}), we get a generalized diffusion equation
\begin{equation}\label{Diffusion_eq}
\rho_t = \nabla \cdot \left( \frac{\rho}{\eta(\rho)} \nabla p(\rho) \right),
\end{equation}
which satisfies the energy-dissipation law (\ref{Diffusion}).

For the PME
\begin{equation}\label{PME_1}
\rho_t = \Delta \rho^{\alpha}, \quad \alpha > 1,
\end{equation}
a commonly used energy-dissipation law is  
\begin{equation}\label{EL_orig}
  \frac{\dd }{\dd t} \int_{\Omega} \frac{1}{\alpha - 1} \rho^{\alpha} =  - \int_{\Omega} \rho |\uvec|^2 \dd x,
\end{equation}
which is consistence with the Wasserstein gradient flows \cite{jordan1998variational, Otto2001PME}. Since $\omega(\rho) = \frac{1}{\alpha - 1} \rho^{\alpha}$ and $\eta(\rho) = \rho$, the force balance gives
\begin{equation}\label{FB_Or}
\rho \uvec = - \nabla \rho^{\alpha},
\end{equation}
which in turns yields the original PME (\ref{PME_1}) by (\ref{Diffusion_eq}).

It should be remarked that same governing equations can be obtained by using different $\omega(\rho)$ and $\eta(\rho)$  [see \cite{duan2019numerical} for examples in the PME]. Correspondingly, different numerical schemes can be derived based on different energy-dissipation laws. 
Besides the classical energy-dissipation law (\ref{EL_orig}) and the two used in (\cite{duan2019numerical}), we can employ another energy-dissipation law
\begin{equation}\label{New_Energy}
  \frac{\dd}{\dd t} \int_{\Omega} \omega(\rho) \dd \x = - \int_{\Omega} |\uvec|^2 \dd \x, \qquad \quad   \omega(\rho) =
  \begin{cases}
    &  2 \rho \ln \rho, \qquad \qquad \qquad  ~~\alpha = 2, \\
    & \dfrac{\alpha}{(\alpha-1)(\alpha-2)} \rho^{\alpha-1}, \quad \alpha >2, \\
  \end{cases}
\end{equation}
for the PME with $\alpha \geq 2$ to construct a numerical scheme. Following the above computation, one can verify that, under the energy-dissipation law (\ref{New_Energy}), the force balance results in
\begin{equation}\label{Traj_new_law}
\uvec = - \nabla  \left( \frac{\alpha}{\alpha - 1} \rho^{\alpha - 1} \right),
\end{equation}
which is equivalent to (\ref{FB_Or}) on the compact support of $\rho_0(\X)$. 
We will show that in the following sections the numerical scheme derived from of (\ref{New_Energy}) has an advantage in tracking the free boundary in the PME, especially in multiple spatial dimensional situation. 
\begin{remark}
Although the energy-dissipation law (\ref{New_Energy}) is only defined for the PME with $\alpha \geq 2$, the force balance (\ref{Traj_new_law}), i.e., the trajectory equation, is well defined for $\alpha > 1$. In the later section, we will show that the numerical scheme derived from of (\ref{New_Energy}) also works well for the case that $1 < \alpha < 2$, although it can not be interpreted through (\ref{New_Energy}) in such cases.
\end{remark}

\begin{remark}
  In this paper, we only consider the diffusions satisfy the energy-dissipation law (\ref{Diffusion}). The free energy in (\ref{Diffusion}) can be generalized to more complicated form. For instance, the famous Cahn-Hilliard equation can be viewed as a $H^1$-diffusion satisfies the energy-dissipation law \cite{Giga2017, liu2019energetic}
  \begin{equation}
\frac{\dd}{\dd t} \int_{\Omega} \frac{1}{2} |\nabla \varphi|^2 + \frac{1}{4 \epsilon^2} (\varphi^2 - 1)^2 = - \int_{\Omega} \phi^2 |\uvec|^2 \dd \x,
  \end{equation}
 where $\varphi$ is the conserved quantity satisfies $\varphi_t + \nabla \cdot (\varphi \uvec) = 0$.

\end{remark}

\section{A discrete energetic variational approach and numerical schemes}

In this section, we first introduce the abstract framework of a discrete energetic variational approach, and apply it to construct some Lagrangian schemes for the porous-medium type generalized diffusion equations. As an example, we construct two numerical schemes for the PME based on energy-dissipation laws (\ref{EL_orig}) and (\ref{New_Energy}).

  As mentioned in the beginning, a discrete energetic variational approach is an analogue to an energetic variational approach in a semidiscrete level. For a given continuous  energy-dissipation law, we can write down a discrete energy-dissipation, 
that is 
\begin{equation}\label{discrte_En_1}
\frac{\dd}{\dd t} E_h \left( \bm{\Xi}(t) \right) = - 2 \mathcal{D}_h(\bm{\Xi}(t), \bm{\Xi}'(t)),
\end{equation}
by introducing a spatial discretization, where $\bm{\Xi}(t) = \left( \Xi_1(t), \Xi_2(t), \ldots, \Xi_K(t) \right) \in \mathbb{R}^K$ is the state variable for the discrete energy-dissipation law (\ref{discrte_En_1}). In the following, we only discuss the case that the kinetic energy $\mathcal{K} = 0$ in the continuous energy-dissipation law. 

Similar to a continuous energetic variational approach stated in sect. 2, we can get the governing equation of $\bm{\Xi}(t)$ by performing LAP and MDP.
In a semidiscrete level, LAP corresponds to taking variation of the discrete action functional
$$\mathcal{A}_h(\bm{\Xi}(t)) = - \int_{0}^T E_h(\bm{\Xi}(t)) \dd t$$ with respect to $\bm{\Xi}(t)$, and MDP corresponds to taking variation of the discrete dissipation functional $\mathcal{D}_h (\bm{\Xi}, \bm{\Xi}(t))$ with respect to $\bm{\Xi}^{\prime}(t)$. Hence, 
the force balance results in:
\begin{equation}\label{discret_D_1}
  \frac{\delta  \mathcal{D}_h  }{\delta \bm{\Xi}'}(\bm{\Xi}(t), \bm{\Xi}'(t)) = \frac{\delta \mathcal{A}_h}{\delta \bm{\Xi}} (\bm{\Xi}(t)),
\end{equation}
which is a system of $K$ nonlinear ordinary differential equations, i.e., ``semi-discrete equations'', of $\bm{\Xi}(t)$. Due to the assumption that the dissipation is quadratic in ``rate'' $\bm{\Xi}'(t)$, equations (\ref{discret_D_1}) can be written as
\begin{equation}\label{discret_D_2}
  {\sf D} \left(\bm{\Xi}(t) \right)  \bm{\Xi}'(t)  = \frac{\delta \mathcal{A}_h}{\delta \bm{\Xi}} \left(\bm{\Xi}(t)\right),
\end{equation}
where $ {\sf D} \left( \bm{\Xi}(t)\right)$ is a $K \times K$ matrix. A numerical scheme can be obtained by introducing a proper temporal discretization to (\ref{discret_D_2}).
If ${\sf D} \left( \bm{\Xi}(t) \right)$ is an identical matrix (which is not true in general),
(\ref{discret_D_2}) can be viewed as a fast descent on each component $\Xi_k$.


A discrete energetic variational approach follows the ``discretize-then-variation'' strategy \cite{furihata2010discrete, christiansen2011topics}. 
The idea of ``discretize-then-variation'', or ``discretize-then-minimize'', has been successfully applied to study numerous problems \cite{furihata2010discrete, christiansen2011topics, davis1998finite, wang2017topological, wang2018formation, carrillo2009numerical, carrillo2016numerical, xu2016variational}. In a recent published book \cite{furihata2010discrete}, \citeauthor{furihata2010discrete} show that ``discretize-then-variation'' is a systematic way to derive a structure-preserving numerical methods for a large class of PDE. By a discrete energetic variational approach, we are able to apply this idea to more general problems with energy-dissipation laws and kinematic relations. In general, ``discrete-then-variation'' and ``variation-then-discrete'' may give us different numerical schemes. It is important to notice that the force balance (\ref{FB}) uses the strong form of the variational results \cite{Giga2017}, since the test functions may be in different space. Hence, the numerical scheme derived by a discrete energetic variational approach normally gives a better approximation to the original energy-dissipation law.

\subsection{Lagrangian schemes for generalized diffusions}
  Now we are ready to construct some Lagrangian schemes for porous medium type generalized diffusions with energy-dissipation law (\ref{Diffusion}) by a discrete energetic variational approach.
The procedure can be extended to more complicated type of diffusions, such as Cahn-Hilliard and PNP equations.

For generalized diffusions considered in this paper, we can discretize the energy-dissipation law by approximating the flow map $\x(\X, t)$ directly, and the value of $\rho(\x, t)$ is determined by the kinematic relation (transport equation) (\ref{rho_detF}), i.e., a composition of the flow map $\x(\X, t)$ and the initial data $\rho_0(\X)$:
\begin{equation}
  \rho(\x, t) = \rho_0 ~\tilde{\circ}~ \X^{-1}(\x, t) = \frac{\rho_0(\X)}{\det F(\X, t)}.
\end{equation}
This is the main difference between our numerical approach and most of traditional Eulerian approaches to diffusion equations.
\begin{remark}
  Although the original free energy possesses certain convexity property with respect to $\rho$, it is not clear whether the convexity is preserved in Lagrangian coordinates, i.e., with respect to the flow map $\x(\X, t)$ in high dimensional situation. 
Hence, for the general case, the evolution of the flow map $\x(\X, t)$ may approach to a ``local'' minimum, which is not necessarily to be the global minimum in theory. 
\end{remark}

In current study, we will discretize $\x(\X, t)$ by a piecewise linear map. An advantage of a piecewise linear approximation to the flow map $\x(\X, t)$ is that the deformation matrix $F(\X,t)$ is piecewise constant (matrix) for give $t$, so are $\det F(\X, t)$ and $F(\X, t)^{-1}$.

One way to construct a piecewise linear approximation to the flow map $\x(\X, t)$ is to use a finite element method \cite{carrillo2018lagrangian}. 
To be more precisely, 
let $\mathcal{T}_h$ be triangulation of $\Omega_0 \in \mathbb{R}^d$. $\mathcal{T}_h$ consists of a set of simplexes $\{ \tau_e ~|~ e = 1,\ldots M \}$ and a set of nodal points $\mathcal{N}_h = \{\X_1, \X_2, \ldots, \X_N \}$. Define the finite element space by
\begin{equation}
V_h = \{  v \in C(\Omega_0) ~|~ v~\text{is linear on each element}~\tau_e \in \mathcal{T}_h  \},
\end{equation}
which is a linear finite element space.
Then the flow map $\x(\X, t)$ can be approximated by
\begin{equation}\label{x_h}
\x_h(\X, t) = \sum_{i = 1}^{N} \avec_i(t) \phi_i(\X) ~\in~ V_h,
\end{equation}
where $\phi_i(\X) : \mathbb{R}^d \rightarrow \mathbb{R}$ is the hat function satisfies $\phi_i(\X_j) = \delta_{ij}$, and $\avec_i(t) \in \mathbb{R}^d$ are coefficients to be determined. Hence,
$$\bm{\Xi}(t) = \left( \xi_1^{(1)}(t), \xi_2^{(1)}(t), \ldots, \xi_N^{(1)}(t), \ldots, \xi_{1}^{(d)},  \xi_{2}^{(d)},  \ldots, \xi_{N}^{(d)} \right)' \in \mathbb{R}^{K},$$
where $K = N \times d$, which is the state variable in (\ref{discrte_En_1}).

Since $\x_h(\X_i, t) = \avec_i$, one can view $\avec_i(t)$ as coordinates in $\Omega_t$ and $\x_h(\X_i, t)$ as a trajectory of a ``particle'' $\X_i$.
The finite element method enables us to compute the deformation matrix $F_h(\X) = \nabla_{X} \x_h $ explicitly for given $\bm{\Xi}$. 
We can write down the deformation matrix $F_h(\X)$ as a $d \times d$-matrix-valued function of $\bm{\Xi}$ on each element $\tau_e$, denoted by
$$F_e(\bm{\Xi}) = \nabla_{\X} \x_h  \big|_{\X \in \tau_e}.$$
The admissible set of $\bm{\Xi}$ is
\begin{equation}
F_{ad}^{\bm{\Xi}} = \left\{ \bm{\Xi} \in \mathbb{R}^K ~|~ \det F_e(\bm{\Xi}) > 0, \quad e = 1, \ldots M \right\}.
\end{equation}
 Correspondingly, the admissible set for $\x_h$ is defined by $$F_{ad}^{\bm{\x}_h} = \left \{ \x_h(\X, t) = \sum_{i = 1}^{N} \avec_i(t) \phi_i(\X) ~|~  \bm{\Xi}(t) \in F_{ad}^{\bm{\Xi}} \right \}.$$
 The non-negativity of $\rho(\x_h(\X, t), t)$ is naturally preserved as if $\x_h(\X, t)$ is in the admissible set $F_{ad}^{\x_h}$. It can be noticed that $F_{ad}^{\bm{\Xi}}$ is not a convex set when $d \geq 2$, which imposes difficulty in numerical analysis.


In the following, we restrict ourselves in a two dimensional case,
the numerical schemes in other spatial dimensions follow easily from this. 
In order to simplify the notation, we let $\avec_i(t) = (a_i(t), b_i(t))$, and denote
$$\bm{a}(t) = (a_1(t), a_2(t), \ldots a_N(t))', \quad \bm{b}(t) = (b_1(t), b_2(t), \ldots, b_N(t))'.$$
Hence, $\bm{\Xi(t)} = (\bm{a}(t), \bm{b}(t))$.
Let $N(i)$ be all the indices $e$ such that $\X_i$ is contained in $\tau_e$ for given $\X_i \in \mathcal{N}_h$. 
The support of $\phi_i$ is denote by $G(i) = \cup_{e \in N(i)} \tau_e$.
For each element $\tau_e$, the nodes of $\tau_e$ is denoted by $\X^e_1$, $\X^e_{2}$, $\X^e_3$,
the global index of the nodes of $\tau_e$ are denoted by $en(e, l)$ ($l = 1, 2, 3$).

Substituting (\ref{x_h}) into (\ref{Diffusion}), 
we get a discrete action functional
\begin{equation}\label{Action_h}
\begin{aligned}
  \mathcal{A}_h (\bm{a}(t), \bm{b}(t)) & = -  \int_{0}^{T} \int_{\Omega_0} \omega \left( \frac{\rho_0(\X)}{\det F(\X, t)} \right) \det F(\X, t) \dd \X \dd t \\
  & = - \int_{0}^T\sum_{e = 1}^{M} \int_{\tau_e}  \omega \left( \frac{\rho_0(\X)}{\det F_e(t)} \right) \det F_e(t) \dd \X \dd t,  \\
\end{aligned}
\end{equation}
and the discrete dissipation
\begin{equation}\label{Dis_h}
\begin{aligned}
  & \mathcal{D}_h(\bm{a}'(t), \bm{b}'(t), \bm{a}(t), \bm{b}(t)) = \frac{1}{2} \int_{\Omega} \eta(\rho) |\x_t|^2 \det F(\X, t) \dd X,  \\
  & = \frac{1}{2} \sum_{e = 1}^M \int_{\tau_e} \eta \left(\frac{\rho_0(\X)}{\det F_e(t)} \right) \left(  \left( \sum_{j = 1}^N a_j'(t) \phi_j(\X) \right)^2 + \left( \sum_{j = 1}^N b_j'(t) \phi_j(\X) \right)^2 \right) \det F_e(t)  \dd \X, \\
\end{aligned}
\end{equation}
where $$F_e(t) := F_e(\bm{a}(t), \bm{b}(t)) = \nabla_{\X} \x_h  |_{\X \in \tau_e}$$  is the deformation matrix $F(\X, t)$ on each element $\tau_e$ at $t$, which can be written down as a function of $a_{en(e, l)}(t)$ and $b_{en(e, l)}(t)$ ($l = 1, 2, 3$) explicitly [see Appendix A]. 

By taking variation of (\ref{Action_h}) with respect to $a_i(t)$ and $b_i(t)$, we have
\begin{equation}\label{dAction_h}
  \begin{aligned}
    & \frac{\delta \mathcal{A}_h}{\delta a_i}(\bm{a}(t), \bm{b}(t)) = -  \sum_{e \in N(i)} \frac{\pp}{\pp a_i} \int_{\tau_e} \omega \left( \frac{\rho_0(X)}{\det F_e(t)} \right) \det F_e(t) \dd \X, \\
    & \frac{\delta \mathcal{A}_h}{\delta b_i}(\bm{a}(t), \bm{b}(t)) = -  \sum_{e \in N(i)} \frac{\pp}{\pp b_i} \int_{\tau_e} \omega \left( \frac{\rho_0(X)}{\det F_e(t)} \right) \det F_e(t) \dd \X, \\
  \end{aligned}
\end{equation}
where
\begin{equation}\label{d_action_el}
  \begin{aligned}
    \frac{\pp}{ \pp \chi} \int_{\tau_e} \omega \left( \frac{\rho_0(X)}{\det F_e} \right) \det F_e \dd X 
    = \int_{\tau_e} \left( - \omega_{\rho} \rho + \omega \right) \left( F^{-\rm{T}}_e : \frac{\pp F_e}{\pp \chi} \right)\det F_e \dd \X, \\
  \end{aligned}
\end{equation}
on each element $\tau_e$, $\chi = a_{en(e, l)}$ or $b_{en(e, l)}$.

On the meanwhile, taking variations of (\ref{discret_D}) with respect to $a_i'(t)$ and $b_i'(t)$ results in
\begin{equation}\label{dDis_h}
  \begin{aligned}
    & \frac{\delta \mathcal{D}_h}{\delta a_i'}  =  \sum_{e \in N(i)} \int_{\tau_e} \eta(\rho) \left( \sum_{j = 1}^N a_j'(t) \phi_j(\X) \right) \phi_i(\X) \det F_e(t)  \dd \X = M_{ij}(\bm{a}(t), \bm{b}(t)) a_j'(t), \\
    & \frac{\delta \mathcal{D}_h}{\delta b_i'}  =  \sum_{e \in N(i)} \int_{\tau_e} \eta(\rho) \left( \sum_{j = 1}^N b_j'(t) \phi_j(\X) \right) \phi_i(\X)  \det F_e(t) \dd \X = M_{ij}(\bm{a}(t), \bm{b}(t)) b_j'(t), \\
  \end{aligned}
\end{equation}
where Einstein summation convention is used, $\eta(\rho) = \eta \left( \rho_0(\X) / \det F_e(t) \right)$, and $M_{ij}(\bm{a}(t), \bm{b}(t))$ is 
defined by
\begin{equation}\label{def_M_2D}
  M_{ij}(\bm{a}(t), \bm{b}(t)) = \sum_{e \in N(i)} \int_{\tau_e} \eta(\rho) \phi_j(\X) \phi_i(\X) \det F_e(t) \dd \X.
\end{equation}


Hence, the force balance equation (\ref{discret_D_1}) results in the ``semi-discrete equations''
\begin{equation}\label{discret_D}
  \begin{aligned}
    & M_{ij}(\bm{a}(t), \bm{b}(t)) a_j'(t) = \frac{\delta \mathcal{A}_h}{\delta a_i}(\bm{a}(t), \bm{b}(t)), \\
    & M_{ij}(\bm{a}(t), \bm{b}(t)) b_j'(t) = \frac{\delta \mathcal{A}_h}{\delta b_i}(\bm{a}(t), \bm{b}(t)), \\
  \end{aligned}
\end{equation}
for $i = 1, 2, \ldots, N$, which can be discretized in time by using some numerical method for systems of ordinary differential equations.  

\begin{remark}
  In this special case, ${\sf D}(\bm{\Xi}(t), \bm{\Xi}'(t))$ in (\ref{discret_D_2}) is given by   
  \begin{equation}\label{def_D}
{\sf D} (\bm{\Xi}(t), \bm{\Xi}'(t)) =
\begin{pmatrix}
  {\sf M} (\bm{a}(t), \bm{b}(t)) & \bm{0} \\
  \bm{0}  & {\sf M}(\bm{a}(t), \bm{b}(t)) \\
\end{pmatrix},
\end{equation}
which is a symmetric matrix. From a computational point of view, $ {\sf D} \left( \bm{\Xi}(t), \bm{\Xi}'(t) \right)$ plays a important role in maintaining the integrity of the flow map in the evolution process.
\end{remark}

Although both (\ref{dAction_h}) and (\ref{dDis_h}) involve the numerical integration over each element, they can be computed out by centroid method (known as midpoint method in 1D). As $F(\X, t)$, $\det F(\X, t)$ and $F(\X, t)^{-1}$ are approximated in a piecewise constant manner, using high-accuracy numerical quadrature over each element cannot improve the numerical accuracy. This is another advantage of the piecewise linear approximation to the flow map, which is actually quadrature-free and can be applied to a high spatial dimensional case. The computational cost is roughly proportional to the number of nodes (``particle''). One can view our numerical approach as a type of cell-centered Lagrangian scheme, where the momentum is defined at the nodes and the other variables (density, pressure, and specific internal energy) are cell-centered \cite{maire2007cell}.  In the following, we denote
$$\rho_e^0 = \rho_0(\X_c^e), \quad \rho_e(t) = \rho_0(\X_c^e) / \det F_e(t), $$
where $\X_c^e$ is the centroid of $\tau_e$.

In order to get a numerical scheme, we need to introduce a proper temporal discretization to the ``semi-discrete equation'' (\ref{discret_D}).
One can use explicit Euler scheme, and the numerical scheme can be written as
\begin{equation}\label{Explict_2D}
  \begin{aligned}
    & M_{ij}^n(\bm{a}^{n}, \bm{b}^n) \frac{a_j^{n+1} - a_j^n}{\tau} = \frac{\delta \mathcal{A}_h}{\delta a_i} (\bm{a}^n, \bm{b}^n), \\
    & M_{ij}^n(\bm{a}^{n}, \bm{b}^n) \frac{b_j^{n+1} - b_j^n}{\tau} = \frac{\delta \mathcal{A}_h}{\delta b_i} (\bm{a}^n, \bm{b}^n), \\
  \end{aligned}
\end{equation}
where
\begin{equation}\label{M_n}
  M_{ij}^{n}(\bm{a}^{n}, \bm{b}^n) = \sum_{e \in N(i)} \int_{\tau_e} \eta(\rho^{n}_e) \phi_j(\X) \phi_i(\X) \det F_e^{n} \dd \X.
\end{equation}
Although the explicit Euler scheme is simple in the numerical implementation, one have to choose  $\tau$ to be significantly small to ensure $\bm{\Xi}^{n+1}_h \in F_{ad}^{\bm{\Xi}}$ and the dissipation of the discrete energy.

A better approach is to adopt a backward Euler scheme for the temporal discretization, i.e.,
\begin{equation}\label{Scheme_2D}
  \begin{aligned}
    & M_{ij}^*(\bm{a}^{n}, \bm{b}^n) \frac{a_j^{n+1} - a_j^n}{\tau} = \frac{\delta \mathcal{A}_h}{\delta a_i}(\bm{a}^{n+1}, \bm{b}^{n+1}), \\
    & M_{ij}^*(\bm{a}^{n}, \bm{b}^n) \frac{b_j^{n+1} - b_j^n}{\tau} = \frac{\delta \mathcal{A}_h}{\delta b_i}(\bm{a}^{n+1}, \bm{b}^{n+1}), \\
  \end{aligned}
\end{equation}
where $M_{ij}^*(\bm{a}^{n}, \bm{b}^n)$ is defined by
\begin{equation}\label{M_star}
M_{ij}^{*}(\bm{a}^{n}, \bm{b}^n) = \sum_{e \in N(i)} \int_{\tau_e} \eta(\rho^{\gamma_1}_e) \phi_j(\X) \phi_i(\X) \det F^{\gamma_2} \dd \X,
\end{equation} 
and $\gamma_i = n$ or $n+1$. Here we have the choice of taking $\eta(\rho_e(t))$ and $\det F_e(t)$ in 
(\ref{def_M_2D}) explicitly or implicitly,  such that ${\sf M}^{*}$ depends $\bm{a}^{n}, \bm{b}^n$, but is independent with $\bm{a}^{n+1}, \bm{b}^{n+1}$.
The theoretical analysis of the choice are in the progress.

  \begin{remark}
    In general, we cannot have an explicit form for the ``semi-discrete equation'' (\ref{discret_D}) and the numerical scheme (\ref{Scheme_2D}) in high dimensional situations, as (\ref{dAction_h}) and (\ref{def_M_2D}) depend on the triangulation $\mathcal{T}_h$. In practice, $M_{ij}(\bm{a}(t), \bm{b}(t))$ and $\frac{\delta \mathcal{A}_h}{\delta a_i}(\bm{a}(t), \bm{b}(t))$ can be assembled using the standard technique in the finite element methods, that is, summing the results on each element over the mesh  \cite{larson2013finite}. [See Appendix B for the procedure of  numerical implementation.]
  \end{remark}

  The numerical scheme (\ref{Scheme_2D}) can be written as
  \begin{equation}\label{Scheme_2D_1}
    {\sf D}^*_n \frac{\bm{\Xi^{n+1}} - \bm{\Xi^{n}}}{\tau} = - \frac{\delta E_h}{\delta \bm{\Xi}} (\bm{\Xi}^{n+1}),
  \end{equation}
  where
   \begin{equation}
  \begin{aligned}
    E_h(\bm{\Xi}) 
     =  \sum_{e = 1}^{M} \int_{\tau_e}  \omega \left( \frac{\rho_e^0}{\det F_e(\bm{\Xi})} \right) \det F_e(\bm{\Xi}) \dd \X, \quad \bm{\Xi} \in F_{ad}^{\bm{\Xi}},
   \end{aligned}
 \end{equation}
 is the discrete energy,  ${\sf D}^*_n = {\sf D}^*(\bm{\Xi}^n)$ is defined by ({\ref{def_D}}) with ${\sf M} = {\sf M}^{*}$. Although (\ref{Scheme_2D_1}) is highly nonlinear equation in
  $$F_{ad}^{\bm{\Xi}} = \left\{ \bm{\Xi} \in \mathbb{R}^K ~|~ \det F_e(\bm{\Xi}) > 0, \quad e = 1, \ldots M \right\},$$
 as an advantage of our discrete energetic variational approach, we can get a solution of (\ref{Scheme_2D_1}) 
 by solving minimization problem in $F_{ad}^{\bm{\Xi}}$:
  \begin{equation}\label{Mini_Problem}
    {\bm{\Xi}}^{n+1} : = \text{argmin}_{\bm{\Xi} \in F_{ad}^{\bm{\Xi}}} J(\bm{\Xi}),
  \end{equation}
  where 
 \begin{equation}\label{min_problem}
   \begin{aligned}
     J(\bm{\Xi}) & =  \frac{1}{2 \tau} {\sf D}^*_n (\bm{\Xi} - \bm{\Xi^{n}}) \cdot (\bm{\Xi} - \bm{\Xi^{n}}) + E_h(\bm{\Xi}).
   \end{aligned}
 \end{equation} 
Moreover,  we can prove the following theorem for the numerical scheme (\ref{Scheme_2D_1}):
\begin{theorem}
  If ${\sf D}^*_n$ is a  symmetric positive-definite matrix, and $\omega$ satisfies
  \begin{equation}
    \lim_{s \rightarrow 0} \omega \left(\frac{1}{s} \right) s = \infty,  
  \end{equation} 
  then for given $\bm{\Xi}^n \in F_{ad}^{\bm{\Xi}}$, there exists a solution $\bm{\Xi}^{n+1}$ to numerical scheme  (\ref{Scheme_2D_1}) such that the following discrete energy dissipation law holds, i.e.,
\begin{equation}\label{d_energy_dis_our_scheme}
  \frac{E_h(\bm{\Xi}^{n+1}) - E_h(\bm{\Xi}^n)}{\tau} \leq  -  \frac{1}{2 \tau^2} {\sf D}_n^{*} (\bm{\Xi}^{n} - \bm{\Xi}^{n+1}) \cdot (\bm{\Xi}^{n} - \bm{\Xi}^{n+1}) \leq 0.
\end{equation}

\end{theorem}

\begin{proof}

Since a minimizer of the minimization problem (\ref{Mini_Problem}) is the solution of (\ref{Scheme_2D_1}), we only need to show that there exists a minimizer of $J(\bm{\Xi})$ in
\begin{equation*}
F_{ad}^{\bm{\Xi}} = \left\{ \bm{\Xi} \in \mathbb{R}^K ~|~ \det F_e(\bm{\Xi}) > 0, \quad e = 1, \ldots M \right\}.
\end{equation*}
Since for $\forall \bm{\Xi} \in \pp F_{ad}^{\bm{\Xi}}$, $\bm{\Xi} = \infty$, following the proof in the Lemma 3.1 in \cite{carrillo2018lagrangian}, the existence of a minimizer can be obtained by showing the set
\begin{equation}
\mathcal{S} = \{ \bm{\Xi} \in F_{ad}^{\bm{\Xi}} ~|~ J(\bm{\Xi}) \leq E_h(\bm{\Xi}^n)  \}
\end{equation}
is a non-empty compact subset of $\mathbb{R}^K$. Obviously, $\bm{\Xi}^n \in \mathcal{S}$, so $\mathcal{S}$ is non-empty.

In order to show $\mathcal{S}$ is compact, we first show that $S$ is bounded. Since ${\sf D}^*_n$ is positive-definite, there exists $\lambda_1 > 0$ such that $\forall \bm{\Xi} \in \mathcal{S}$
\begin{equation}
|| \bm{\Xi} - \bm{\Xi}^n ||^2 \leq \frac{1}{\lambda_1}  {\sf D}^*_n (\bm{\Xi} - \bm{\Xi^{n}}) \cdot (\bm{\Xi} - \bm{\Xi^{n}}) \leq  \frac{2 \tau}{\lambda_1} \left(E_h(\bm{\Xi}^n) - E_h(\bm{\Xi}) \right)
\end{equation}
Hence, $\mathcal{S}$ is bounded. Next we show $\mathcal{S}$ is closed in $\mathbb{R}^{K}$, it suffices to show that the limit $\widetilde{\bm{\Xi}}$ for any sequence $(\bm{\Xi}^{(k)})_{k=1}^{\infty} \subset \mathcal{S}$ is in $F_{ad}^{\bm{\Xi}}$. For $\forall e \in \{ 1, 2, \ldots M \}$ and all $k$
\begin{equation}
  \begin{aligned}
    E_h(\bm{\Xi}^n) \geq E_h(\bm{\Xi}^{(k)}) \geq  \omega \left( \frac{\rho_e^0}{\det F_e(\bm{\Xi}^{(k)})} \right) \det F_e(\bm{\Xi}^{(k)}) |\tau_e|, 
  \end{aligned}
\end{equation}
where $|\tau_e|$ is the area of element $\tau_e$. Since $ \omega \left( \frac{\rho_e^0}{\det F_e(\bm{\Xi}^{(k)})} \right) \det F_e(\bm{\Xi}^{(k)}) \rightarrow \infty$ if $\det F_e \rightarrow 0$, we can conclude that $\det F_e(\bm{\Xi}^{(k)}) > 0$ is uniformly bounded away from zero. So $\det F_e(\bar{\bm{\Xi}}) > 0$ for all $e$, which means $\widetilde{\bm{\Xi}} \in F_{ad}^{\bm{\Xi}}$.

If $\bm{\Xi}^{n+1}$ is a minimizer in $\mathcal{S}$, we have
\begin{equation}
\frac{1}{2 \tau} {\sf D}^*_n (\bm{\Xi}^{n+1} - \bm{\Xi^{n}}) \cdot (\bm{\Xi} - \bm{\Xi^{n}}) + E_h(\bm{\Xi}^{n+1}) \leq E_h (\bm{\Xi}^{n}),
\end{equation}
which completes the proof.
\end{proof}

\begin{remark}
  In one-dimensional case, due to  $\det F = F$, it is easy to show that $E_h(\bm{\Xi})$ is convex. Hence, we can have a stronger result similar to that in \cite{duan2019numerical}:
  \begin{equation}\label{Energy_dis_1D}
  \begin{aligned} 
    \frac{E_h(\bm{\Xi}^{n}) - E_h(\bm{\Xi}^{n+1})}{\tau} & \geq  \frac{\delta E_h}{\delta \bm{\Xi}} (\bm{\Xi}^{n+1})  \cdot \frac{\bm{\Xi}^{n} - \bm{\Xi}^{n+1}}{\tau}  \\
    & = \frac{1}{\tau^2} {\sf D}^{*}_n  (\bm{\Xi}^{n+1} - \bm{\Xi}^{n}) \cdot (\bm{\Xi}^{n+1} - \bm{\Xi}^{n}) \geq 0, \\
  \end{aligned}
  \end{equation}
  where the first inequality follows the convexity of $E_h(\bm{\Xi})$. Moreover, since $ \bar{F}_{ad}^{\bm{\Xi}}$ is the closed convex set in 1D \cite{duan2019numerical}, we can have a uniquely solvable result of numerical scheme (\ref{Scheme_2D_1}).

  In high dimensional situation ($d \geq 2$), we can not get a uniquely solvable result of numerical scheme (\ref{Scheme_2D_1}) due to the lack of convexity of $J(\bm{\Xi})$ and $E(\bm{\Xi})$ \cite{carrillo2018lagrangian}. 
\end{remark}

\begin{remark}
  Although there exists a (local) minimizer in $ \mathcal{S} \subset F_{ad}^{\bm{\Xi}}$ for $J(\bm{\Xi})$, which is a solution of numerical scheme (\ref{Scheme_2D}) decreases the discrete energy $E_h$, we still need to choose a proper optimization methods to find a minimizer $\bm{\Xi}^{n+1} \in \mathcal{S}$. For the PME in 1D or 2D, numerical tests show that a standard damped Newton's method with fixed step-size is adequate to this purpose. For more general problem, we can adopt optimizaition methods wih line search to find such a minimizer. Since we may only find a local minimizer of $J(\bm{\Xi})$ in $\mathcal{S}$ due to lack of convexity for the general problem, the dynamical evolution of a flow-map based Lagrangian methods might be different from Eulerian methods.
\end{remark}

  \begin{remark}
 The above approach start with the spatial discretization to the flow map (\ref{x_h}), we can also begin with introducing a temporal discretization to the continuous energy-dissipation law (\ref{Diffusion}) by
  \begin{equation}\label{d_energy_dis}
    \frac{E(\x^{n+1}) - E(\x^{n})}{\tau} = - \int_{\Omega} \eta(\rho^{\gamma_1}) \Big|\frac{\x^{n+1}_h - \x^{n}_h}{\tau}\Big|^2 \det F_h^{\gamma_2} \dd \X.
  \end{equation}
  where $\x^{n} = \x(\X, t_n)$, and $\gamma_i = n$ or $n+1$, as in (\ref{M_star}). The RHS of (\ref{d_energy_dis}) can be viewed as an approximation to $\frac{\dd}{\dd t} E(\x(\X, t^{*}))$  for some $t^{*} \in [t^n, t^{n+1}]$, we obtain (\ref{d_energy_dis}) by further approximating 
  \begin{equation}\label{d_dis}
    - \int_{\Omega} \eta(\rho(t^{*})) |\x_t(t^*)|^2 \det F_h(t^{*}) \dd \X \approx  - \int_{\Omega} \eta(\rho^{\gamma_1}) \Big|\frac{\x^{n+1} - \x^{n}}{\tau}\Big|^2 \det F_h^{\gamma_2} \dd \X.
  \end{equation}
Using the spatial discretization (\ref{x_h}), we will have the following discrete energy-dissipation law:
  \begin{equation}\label{new_scheme_2}
    \frac{E_h(\bm{\Xi}^{n+1}) - E_h(\bm{\Xi}^{n})}{\tau} =  - {\sf D}^{*}_n \frac{\bm{\Xi}^{n+1} - \bm{\Xi}^{n}}{\tau} \cdot  \frac{\bm{\Xi}^{n+1} - \bm{\Xi}^{n}}{\tau}, 
  \end{equation}
  where $ {\sf D}^{*}_n = {\sf D}^{*}(\bm{\Xi}_n)$ is same to that in (\ref{Scheme_2D_1}).  Note that
  \begin{equation}
   E_h(\bm{\Xi}^{n+1}) - E_h(\bm{\Xi}^{n}) = \nabla_{\bm{\Xi}} E_h (\bm{\Xi}_c) \cdot (\bm{\Xi}^{n+1} - \bm{\Xi}^n)
  \end{equation}
  for some $\bm{\Xi}^{c} = (1 - c) \bm{\Xi}^{n+1} + c \bm{\Xi}^{n}$. It is straightforward to show that if $\bm{\Xi}^{n+1}  \in  F_{ad}^{\bm{\Xi}}$ satisfies
  \begin{equation}
    {\sf D}^{*}_n \frac{\bm{\Xi}^{n+1} - \bm{\Xi}^{n}}{\tau} = - \nabla_{\bm{\Xi}} E_h(\bm{\Xi}_c),
  \end{equation}
  then $\bm{\Xi}^{n+1}$ satisfies (\ref{new_scheme_2}).

  In our scheme (\ref{Scheme_2D_1}), we approximate $\nabla_{\bm{\Xi}} E_h (\bm{\Xi}_c)$ by $\nabla_{\bm{\Xi}} E_h (\bm{\Xi}^{n+1})$, which causes the difference between (\ref{d_energy_dis_our_scheme}) and (\ref{new_scheme_2}).
  \end{remark}

  \begin{remark}
    We can design a numerical scheme that satisfies (\ref{new_scheme_2}) exactly. Let
    \begin{equation}
      \bm{\Xi}^{n} = \sum_{i = 1}^K \kappa_i^{n} \bm{e}_i, \quad \bm{\Xi}^{n+1} = \sum_{i = 1}^K \kappa_i^{n+1} \bm{e}_i,
    \end{equation}
    where $\bm{e}_i$ is the standard orthonormal basis in $\mathbb{R}^K$.
    Note the RHS of (\ref{new_scheme_2}) can be written as
    \begin{equation}
      - \frac{1}{\tau^2}\sum_{i=1}^K \left(\sum_{j = 1}^K D_{ij}^{*}(\bm{\Xi}^n) (\kappa_j^{n+1} - \kappa_j^n) \right) (\kappa_i^{n+1} - \kappa_i^n).
    \end{equation}
    On the other hand, the LHS of (\ref{new_scheme_2}) can be written as
    \begin{equation}
     \frac{1}{\tau} \sum_{i = 1}^{K} E_h(\bm{\Xi}^n_{(i)}) - E_h(\bm{\Xi}^n_{(i-1)}),
    \end{equation}
    where
    \begin{equation}
      \bm{\Xi}^n_{(0)} = \bm{\Xi}^n, \quad \bm{\Xi}^n_{(i)} = \bm{\Xi}^{n} + \sum_{l = 1}^i (\kappa_l^{n+1} - \kappa_l^{n}) \bm{e}_l.
    \end{equation}
    Hence, direct computation shows that $\bm{\Xi}^{n+1} = \bm{\Xi}^n_{(K)}$ satisfies (\ref{new_scheme_2}) if
    \begin{equation}
      E_h(\bm{\Xi}^n_{(i)}) - E_h(\bm{\Xi}^n_{(i-1)}) = - \frac{1}{\tau}\sum_{i=1}^K \left(\sum_{j = 1}^K D_{ij}^{*}(\bm{\Xi}^n) (\kappa_j^{n+1} - \kappa_j^n) \right) (\kappa_i^{n+1} - \kappa_i^n), \quad i = 1, \ldots, K,
    \end{equation}
    which give us a numerical scheme
    \begin{equation}\label{NewScheme}
        \frac{1}{\tau}\sum_{j = 1}^K D_{ij}^{*}(\bm{\Xi}^{n}) (\kappa_j^{n+1} - \kappa_j^n)  =  -  \frac{ E_h(\bm{\Xi}^n_{(i)}) - E_h(\bm{\Xi}^n_{(i-1)}) }{\kappa_i^{n+1} - \kappa_i^n}, \quad i = 1, \ldots, K.
    \end{equation}
    The scheme (\ref{NewScheme}) preserves the discrete energy-dissipation law (\ref{new_scheme_2}), and might be useful when the variation of $E(\bm{\Xi})$ ($\mathcal{A}(\bm{\Xi})$) cannot be computed efficiently.
However, to our knowledge, the numerical analysis and experiments for such type of scheme is lacking.
We will explore this type of scheme in the future work.    
  \end{remark}

\begin{remark}
  One can also adopt some high-order temporal discretization to the ``semi-discrete equations'' (\ref{discret_D}). Since the resulting numerical scheme might be complicated to deal with, we will study the high-order temporal discretization in the future work.
\end{remark}

\begin{remark}
  In general, discrete-then-variation and variation-then-discrete may give us different numerical scheme. In order to get the numerical schemes (\ref{Explict_2D}) and (\ref{Scheme_2D}), one should substitute (\ref{x_h}) into a particular weak form of the force balance (\ref{Traj_PME}) (strong form of the variational results), and introduce a proper approximation and temporal discretization. A weak form of (\ref{Traj_PME}) can be written as
  \begin{equation}\label{strong_weak}
    \int_{\Omega_0} \eta(\rho) \x_t \cdot \y \det F  \dd \X  = - \int_{\Omega_0} \left( - \omega_{\rho} \rho + \omega \right) (F^{-\rm{T}} : \nabla_{\X} \y) \det F \dd \X,
  \end{equation}
  where $\y$ is a test function, $F = \nabla_{\X} \x$. %
One can get (\ref{Scheme_2D}) by taking the test function $\y = \phi_i$ ($i = 1, \ldots N$) and approximating (\ref{strong_weak}) by
\begin{equation}\label{Weak_FL}
  \int_{\Omega_0} \eta(\rho^{\gamma_1})  \frac{\x_h^{n+1} - \x_h^{n}}{\tau} \cdot \y \det F^{\gamma_2} \dd \X =  - \int_{\Omega_0} \left( - \omega_{\rho} \rho + \omega \right) (F^{-\rm{T}} : \nabla_{\X} \y) \det F^{n+1} \dd \X.
\end{equation}
It should be remarked that we might need to approximate $\det F$ in both side of (\ref{strong_weak}) in different manners (explicitly or implicitly) to get back to (\ref{Scheme_2D}).
    From an energetic variational approach point of view, the test functions may be in different space for the LAP and the MDP. Hence, starting with the force balance (\ref{Traj_PME}) may not give us a structure-preserving Lagrangian scheme without using a proper weak form.
\end{remark}

At the end of this subsection, we briefly talk about the post-process after we obtain the discrete flow map $\x_h(\X, \x)$. According to the kinematic relation, $\rho(\x)$ can be computed by
\begin{equation}
\rho(\x(\X, t), t) = \rho_0(\X) / \det F(\X, t).
\end{equation}
Hence, we can compute the density on each element, i.e, the density in the centroid of each element is by
\begin{equation}\label{rho_e_c}
\rho_h(\x_c^e, t) = \rho_0(\X_c^e) / \det F_e(t),
\end{equation}
where $\x_c^e$ is the centroid of $\x_h(\tau_e)$, while $\X_c^e$ is the centroid of $\tau_e$. For each node $X_i$, since the determinant of the deformation matrix $F$ may be different in different elements contain $\X_i$, 
we can compute the $\rho_h(\x_h(\X_i), t)$ in each nodes by
\begin{equation}
\rho_h(\x_h(\X_i), t) = \rho_0(\X_i) \frac{\sum_{e \in G(i)} |\tau_e|}{\sum_{e \in G(i)} |\tau_e| \det F_e(t)}.
\end{equation}

\subsection{Application in the PME}

The above framework works for any porous medium type generalized diffusions has the energy-dissipation law (\ref{Diffusion}) and the kinematic relation (\ref{Transport}). Next, we apply such framework to develop two numerical schemes for the PME based on energy-dissipation law (\ref{EL_orig}) and (\ref{New_Energy}). The RHS of (\ref{Scheme_2D}) can be computed from (\ref{dAction_h}) and (\ref{d_action_el}) for a given $\omega(\rho)$.  For the LHS in (\ref{Scheme_2D}), we take
\begin{equation}\label{PME_scheme_1}
  M_{ij}^{*}(\bm{a}^{n}, \bm{b}^n) = \sum_{e \in N(i)} \int_{\tau_e} \rho_0(\X) \phi_j(\X) \phi_i(\X) \dd \X,
\end{equation}
for the PME with energy-dissipation law (\ref{EL_orig}), while for the energy-dissipation law (\ref{New_Energy}), we take
\begin{equation}\label{PME_scheme_2}
  M_{ij}^{*}(\bm{a}^{n}, \bm{b}^n) = \sum_{e \in N(i)} \int_{\tau_e}  \phi_j(\X) \phi_i(\X) \det F^{n} \dd \X.
\end{equation}
We call the numerical scheme based on energy-dissipation law (\ref{EL_orig}) as the scheme 1 [(\ref{Scheme_2D}) with (\ref{PME_scheme_1})] and the numerical scheme based on energy-dissipation law (\ref{New_Energy}) as scheme 2 [(\ref{Scheme_2D}) with (\ref{PME_scheme_2})]. One can develop other numerical schemes for the PME by other energy-dissipation laws, such as two used in Ref. \cite{duan2019numerical}.

Due to the degeneracy of the PME at $\rho = 0$, the semi-discrete equation (\ref{discret_D_1}), corresponding to the energy-dissipation law (\ref{EL_orig}), is also degenerate at the region that $\rho_0(\X) = 0$, that is to say, if $\X_k \in \mathbb{R}^2 \backslash \overline{\Omega}_c$, where $\Omega_c$ is the support of $\rho_0(\X)$, then
\begin{equation}
  M_{kj}(\bm{a}(t), \bm{b}(t)) = 0, \quad \frac{\delta \mathcal{A}_h}{\delta a_k}(\bm{a}(t), \bm{b}(t)) = 0, \quad \frac{\delta \mathcal{A}_h}{\delta b_k}(\bm{a}(t), \bm{b}(t)) = 0,
\end{equation}
which means the nodes (``particle'') in such region have no well-defined velocity. 
In the meantime, we can  only derive the PME from the energy-dissipation law (\ref{New_Energy}) on the compact support of $\rho_0(\X)$.
Hence, in the following, we always assume that $\Omega_0$ is the compact support of $\rho_0(\X)$ or $\rho_0(\X) > 0$ on $\Omega_0$, which also guarantees the conditions of Theorem 3.1 hold.

It is a commonly used strategy that solving the PME only within the solution support \cite{tomoeda1983numerical, dibenedetto1984interface, hoff1985linearly, bertsch1990numerical, budd1999self, baines2005moving, ngo2018adaptive, duan2019numerical}, known as the non-embedding approach.
The main challenge of this approach is that the evolution of free boundary has to be tracked explicitly \cite{ngo2018adaptive}.
Our cell-centered Lagrangian schemes enable us to treat the movement of free boundary in a uniform way, as the velocity of the free boundary is also well-defined, which is a a major advantage in high dimensional situations. In later section, we will show that our schemes, especially the scheme based on energy-dissipation law (\ref{New_Energy}) can capture the free boundary of the PME, in both 1D and 2D, without explicitly track the movement of the free boundary.   

\begin{remark}
  Since the corresponding trajectory equation (\ref{Traj_new_law}) of (\ref{New_Energy}) is exactly the equation for the movement of free boundary, we expect
the numerical scheme based on the energy-dissipation law (\ref{New_Energy}) has an advantage in tracking the movement of the free boundary.
 Although the energy-dissipation law (\ref{New_Energy}) is only valid for the PME with $\alpha \geq 2$, our numerical tests show that the scheme 2 also works well for $1 < \alpha < 2$.
  \end{remark}





{\noindent \bf Numerical schemes in 1D:} We can write down our two numerical schemes, based on energy-dissipation law (\ref{EL_orig}) and (\ref{New_Energy}), explicitly in one-dimensional case.
Let $\Omega_0 = [\xi_l, \xi_r]$ be the compact support of $\rho_0(X)$, and  $\xi_l = X_1 < X_2  < \ldots < X_{N} = \xi_r$ be nodes in $\Omega_0$.
We can approximate the flow map $x(X, t)$ by
\begin{equation}
x_h(X, t) = \sum_{i = 1}^{N} a_i(t)  \phi_i(X),
\end{equation}
where $\phi_i(X)$ is a hat function satisfies $\phi_i(X_j) = \delta_{ij}$.
 Let $\bm{a}(t) = (a_1(t), a_2(t), \ldots, a_{N}(t))'$ and $h_i = X_{i+1} - X_i$, 
in one-dimensional case, the admissible set of $\bm{a}$ is simply as
\begin{equation}
F_{ad}^{\bm{a}} = \left \{ \bm{a} = (a_1, a_2, \ldots, a_N)' ~|~ a_1 < a_2 < \ldots < a_N   \right \}.
\end{equation}
The discrete action functional and the discrete dissipation can be written as
\begin{equation}\label{discret_E_1D}
  \begin{aligned}
  & \mathcal{A}_h (\bm{a}(t))  = -  \int_{0}^{T} \sum_{i = 1}^{N-1} \int_{X_i}^{X_{i+1}}  \omega \left( \frac{ \rho_0(X)}{ \det F_i(t) } \right) \det F_i(t) \dd X \dd t, \\
    & \mathcal{D}_h(\bm{a}'(t), \bm{a}(t)) = \frac{1}{2} \sum_{i = 1}^{N-1} \int_{X_i}^{X_{i+1}} \eta \left(\frac{\rho_0(X)}{\det F_i(t)} \right)  \left( \sum_{j = 1}^N a_j'(t) \phi_j(X) \right)^2 \det F_i(t)  \dd \X, \\
  \end{aligned}
  \end{equation}
  where $$\det F_i = (a_{i+1} - a_i)/h_i, \quad i = 1, 2, \ldots N - 1.$$

  For the energy-dissipation law (\ref{EL_orig}), direct computation results in
\begin{equation}
  \begin{aligned}
    \frac{\delta \mathcal{A}_h}{\delta a_i}  (\bm{a}) & =  - \frac{1}{h} \left(  \int_{X_i}^{X_{i+1}} \left(  \frac{\rho_0(X)}{(a_{i+1} - a_i) / h_i}  \right)^{\alpha}   \dd X   - \int_{X_{i-1}}^{X_{i}} \left(  \frac{\rho_0(X)}{(a_i - a_{i - 1}) / h_{i-1}}  \right)^{\alpha}   \dd X \right) \\
                                     & \approx  -  \left( \left(  \frac{\rho_0(X_{i+1/2})}{(a_{i+1} - a_{i}) / h_{i}}  \right)^{\alpha}  -  \left(  \frac{\rho_0(X_{i-1/2})}{(a_i - a_{i - 1}) / h_{i-1}}  \right)^{\alpha}   \right), \\
  \end{aligned}
\end{equation}
and
\begin{equation}
  \begin{aligned}
  \frac{\delta \mathcal{D}_h}{\delta a_i'(t)} (\bm{a}(t), \bm{a}'(t))
    & =   \int_{X_{i-1}} ^{X_{i}} \rho_0(X) \left( \sum_{j = 1}^N a_j'(t) \phi_j(X) \right) \phi_i(X)  \dd \X  \\
    & \quad + \int_{X_i} ^{X_{i+1}} \rho_0(X) \left( \sum_{j = 1}^N a_j'(t) \phi_j(X) \right) \phi_i(X)  \dd \X. \\
    \end{aligned} 
\end{equation}
Hence, in the one-dimensional case, our scheme 1 based on energy-dissipation law (\ref{EL_orig}) can be written as
\begin{equation}\label{scheme1}
M_{ij} \frac{a_{j}^{n+1} - a_{j}^n}{\tau} = -  \left( \left(  \frac{\rho_0(X_{i+1/2})}{(a_{i+1}^{n+1} - a_{i}^{n+1}) / h_i}  \right)^{\alpha}  -  \left(  \frac{\rho_0(X_{i-1/2})}{(a_i^{n+1} - a_{i - 1}^{n+1}) / h_{i-1}}  \right)^{\alpha}   \right)
\end{equation}
where ${\sf M}$ is a triangular matrix, given by
\begin{equation}
  M_{ij} =
  \begin{cases}
    \rho_0(X_{i - 1/2}) h_{i-1} /6, \quad & j  = i - 1,  \\
    \rho_0(X_{i - 1/2}) h_{i-1}/3 +  \rho_0(X_{i + 1/2}) h_{i}/3, \quad & j  = i,  \\
    \rho_0(X_{i + 1/2}) h_i/6, \quad & j  = i + 1,  \\
    0, \quad & \text{otherwise},  \\
  \end{cases}  \quad    1 \leq i, j \leq N,
\end{equation}
where we define $X_{1/2} = 0$, $X_{N+1/2} = 0$ and $h_0 = h_N = h$ to simplify the notation.

\begin{remark}
  In Ref. \cite{duan2019numerical}, the authors develop several numerical schemes for the one-dimensional PME based on different energy-dissipation laws. 
  For the energy-dissipation law (\ref{EL_orig}), their numerical scheme can be written as (scheme 0 in \cite{duan2019numerical})
\begin{equation}
 \rho_0(X_i) \frac{a_{i}^{n+1} - a_{i}^n}{\tau} = -  \frac{1}{h}\left( \left(  \frac{\rho_0(X_{i+1/2})}{(a_i^{n+1} - a_{i - 1}^{n+1}) / h}  \right)^{\alpha}  -  \left(  \frac{\rho_0(X_{i-1/2})}{(a_i^{n+1} - a_{i - 1}^{n+1}) / h}  \right)^{\alpha}   \right).
\end{equation}
Formally, with the equidistant node, our scheme (\ref{scheme1}) only differs from theirs in the temporal discretization for all inner points. 
 A drawback of their temporal discretization is that, in order to prevent the solution to escape from the admissible set, they should use a specific and inefficient damped Newton methods, in which the step-size is roughly proportional to $\min(\rho_0(X_i))$. Our temporal discretization follow the maximum dissipation principle, in which $M(\bm{a})$ control the movement of each nodes. A standard damped Newton with a fixed step-size can prevent the solution to escape from the admissible set in our scheme. Their scheme can only be applied to the region in which $\rho_0 > 0$ (a different scheme is used for the movement of free boundary).
Moreover, the starting point of the numerical approach in Ref. \cite{duan2019numerical} is the force balance (\ref{FB_Or}) in strong form,
it is difficult to extend their approach to a high dimensional situation. Strictly speaking, their numerical schemes  may not preserve the original energy-dissipation law in general.
\end{remark}

Similarly, in the one-dimensional case,  our scheme 2, which is based on energy-dissipation law (\ref{New_Energy}), can be written as
\begin{equation}\label{scheme2}
M_{ij}(\bm{a}^n)\frac{x_{j}^{n+1} - x_{j}^n}{\tau} =  - \frac{\alpha}{\alpha - 1}\left( \left(  \frac{\rho_0(X_{i+1/2})}{(a_i^{n+1} - a_{i - 1}^{n+1}) / h}  \right)^{\alpha-1}  -  \left(  \frac{\rho_0(X_{i-1/2})}{(a_i^{n+1} - a_{i - 1}^{n+1}) / h}  \right)^{\alpha-1}   \right),
\end{equation}
where
\begin{equation}
  M_{ij}(\bm{a}^n) =
  \begin{cases}
    (a_{i}^n - a_{i-1}^n)/6, \quad & j  = i - 1,  \\
    (1 - \delta_{1i})(a_{i}^n - a_{i-1}^n)/3 +  (1 - \delta_{Ni})(a_{i+1} - a_{i})/3, & \quad j  = i,  \\
    (a_{i+1}^n - a_i^n)/6, \quad & j  = i + 1,  \\
    0, \quad & \text{otherwise},  \\
  \end{cases}  \qquad    1 \leq i, j \leq N.
\end{equation}
where we define $X_{1/2} = 0$, $X_{N+1/2} = 0$ and $h_0 = h_N = h$ to simplify the notation.

\section{Numerical Experiments}
In this section, we present some numerical results to the PME in both 1D and 2D to demonstrate the accuracy of our numerical methods. 
We'll focus on our scheme 2, which is based on the energy-dissipation law (\ref{New_Energy}).
The error in space of a numerical solution is measured in $L^2$-norms defined by
  \begin{equation}\label{def_L2}
   ||e_h||^2_2 = \left(\int_{\Omega} e_h^2(\x) \dd \x \right)^{1/2},
  \end{equation}
where $e_h(\x)$ is the difference between the numerical solution $\rho_h(\x)$ and the exact solution $\rho(\x)$.
We compute the numerical integration in (\ref{def_L2}) by the centroid method (known as midpoint method in 1D), which defines the discrete $\mathcal{L}^2$-norm, i.e.,
  \begin{equation}
    ||e_h||^2_{\mathcal{L}^2} =  \left( \sum_{e = 1}^{M} e_h(\x_c^e)^2 |\tau_e| \det F_e  \right)^{1/2},
  \end{equation}
  where $\x_c^e$ is the centroid of $\x_h(\tau_e)$, $\rho_h(\x_c^e)$ is computed by (\ref{rho_e_c}). 

\subsection{One-dimensional problems}
\subsubsection{Barenblatt-Pattle solution}
To verify the accuracy of our numerical methods, we first consider a benchmark solution, the Barenblatt-Pattle solution, which is a exact weak solution for the PME established by Barenblatt \cite{gi1952some} and Pattle \cite{pattle1959diffusion}. The one-dimensional Barenblatt-Pattle solution is given by,
\begin{equation}
B_{\alpha} (x ,t) = t^{-k} \left[ \left(1 - \frac{k(\alpha-1)}{2\alpha} \frac{|x|^2}{t^{2k}}\right)_{+}  \right]^{1/(\alpha-1)}, \quad t > 0,
\end{equation}
where $k = (\alpha + 1)^{-1}$ and $u_{+} = \mathrm{max} (u, 0)$. For any time $t > 0$, this solution has a compact support $[-\xi_{\alpha}(t), \xi_{\alpha}(t)]$ with the interface $|x| = \xi_{\alpha}$ moving outward at a finite speed, where
\begin{equation}
\xi_{\alpha}(t) = \sqrt{\frac{2 \alpha}{k (\alpha - 1)}} t^{k}.
\end{equation}

We take the Barenblatt solution at $t = 1$ , i.e., $B_{\alpha}(x, 1)$, as the initial data, and compare our numerical solution at time $T$ with $B_{\alpha}(x, T+1)$.
Fig. \ref{BP_1D} shows the numerical and exact solutions for $\alpha = 4$ at $T = 1$ and $T = 10$, where the numerical solutions are computed by scheme 1 (\ref{scheme1}) [shown by blue-square] and scheme 2 (\ref{scheme2}) [shown by red-circle] with $\Omega_0$ to be the compact support of the initial data. The results demonstrate that both our numerical schemes can approximate the exact solutions well without oscillation. The numerical solutions by scheme 2 (\ref{scheme2}) approximate the exact solutions better near the interface.  
\begin{figure}[!h]
    \includegraphics[width = \linewidth]{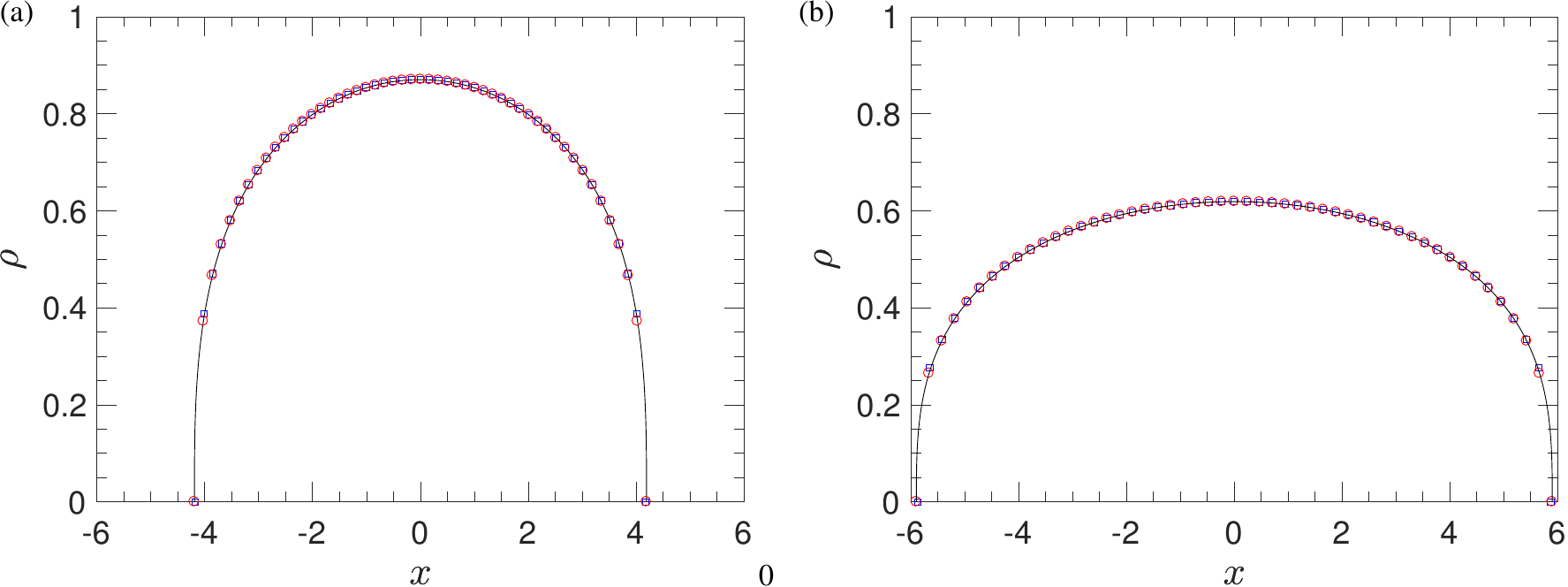}
    \caption{Solution for the PME for $\alpha = 4$, with $\rho_0(X) = B_{\alpha}(x, 1)$ at different time: (a) $T = 1$,  (b) $T = 10$ ($N = 51$, $h \approx 0.14$, $\tau = 0.01$). [Exact: dark line, scheme 1: blue square, scheme 2 : red circle].  }\label{BP_1D} 
\end{figure}



The converge rate of Barenblatt solutions with $\alpha = 3$ and $4$ for both numerical schemes is shown in Table \ref{Table1}. The error at $X = 0$ and the error in $\mathcal{L}^2$-norm is presented. 
For both schemes, the converge rate of the error at $X = 0$ is second order in space since the solution is smooth far away from the interface. The scheme 1 has a first-order converge rate, while the scheme 2 can achieve second-order in $\mathcal{L}^2$-norm at $T = 1$.
When $\alpha$ becomes larger, the converge rate of scheme 2 in $\mathcal{L}^2$-norm can keep in second order.
As expected, the numerical error of scheme 2  is smaller than that of scheme 1, as it track the movement of free boundary better. We can reduce the numerical error of our scheme 1 by tracking the movement of the free boundary explicitly, i.e., replacing the equations of $X_b \in \pp \Omega_0$ with the equation of free boundary.
In the following, we'll focus on our scheme 2, which based on energy-dissipation law (\ref{New_Energy}). All the following numerical solutions are computed by scheme 2.
\begin{table}[!h]
  \begin{center}
    {\scriptsize
  \begin{tabular}{ c | c | c  c  c  c | c  c  c  c }
    \hline
           \multicolumn{2}{c|}{$\alpha = 3$}         &  \multicolumn{4}{c|}{scheme 1} &  \multicolumn{4}{c}{scheme 2}  \\ \hline
                   N         &  $\tau$ &  Error at $X(0)$ & Order   & $L_m^2$-error & Order &  Error at $X(0)$ & Order  & $L_m^2$-error & Order \\ \hline
            $51$    & 1/100  &  2.6421e-04  &         & 0.0036  &         & 9.0421e-05  &         &  3.5109e-04 &         \\ \hline
            $101$   & 1/400  & 7.0619e-05   & 1.9036  & 0.0016  & 1.1699  & 2.2692e-05  & 1.9945  & 9.5494e-05  &   1.8784       \\  \hline  
            $201$   & 1/1600 & 1.9303e-05 & 1.8712  & 7.0440e-04  & 1.1836  & 5.7058e-06  & 1.9917  & 2.5714e-05  & 1.8929   \\  \hline       
           \multicolumn{2}{c|}{$\alpha = 4$}         &  \multicolumn{4}{c|}{scheme 1} &  \multicolumn{4}{c}{scheme 2}  \\ \hline
                   N         &  $\tau$ & Error at $X(0)$ & Order   & $L_m^2$-error & Order &  Error at $X(0)$ & Order  & $L_m^2$-error & Order \\ \hline
            $51$    & 1/100  &  2.6925e-04  &           & 0.0051 &         &   2.3969e-04  &          & 5.0786e-04  &         \\ \hline
            $101$   & 1/400  &  7.6531e-05  &  1.8148   & 0.0027 & 0.9175  &   6.0120e-05  &  1.9953   & 1.4894e-04  & 1.7697   \\  \hline  
            $201$  & 1/1600 &   2.2855e-05  &  1.7435   & 0.0014 & 0.9475  &   1.5073e-05  &  1.9959    & 4.3149e-05  & 1.7873  \\  \hline       
  \end{tabular}
  }
  \caption{The convergence rate of numerical solutions for $\rho_0(X) = B_{\alpha}(X, 1)$ at the finite time T = 1 for $\alpha = 2$ and $\alpha = 4$.}
  \label{Table1}
\end{center}
\end{table}

Fig. \ref{BP_Tra}(a) shows the trajectory of each node for Barenblatt solution with $\alpha = 4$ [$N = 51$, $\tau = 0.01$, scheme 2]. It can be noticed that the final grid is almost uniform. This is because the solution doesn't become steeper during the the time evolution.
We plot in Fig. (\ref{BP_Tra})(b) the evolution of the numerical interface for the Barenblatt solution [$N = 51$, $\tau = 0.01$, scheme 2], with four different parameters $\alpha = 4, 5, 6$ and $8$, from $T = 0$ to $1$, in which the solid line is the position of the exact interface, and the circle indicates the position of the numerical interface. The results show that the exact interface can be approximated in high accuracy.  
\begin{figure}[!h]

   \includegraphics[width = \linewidth]{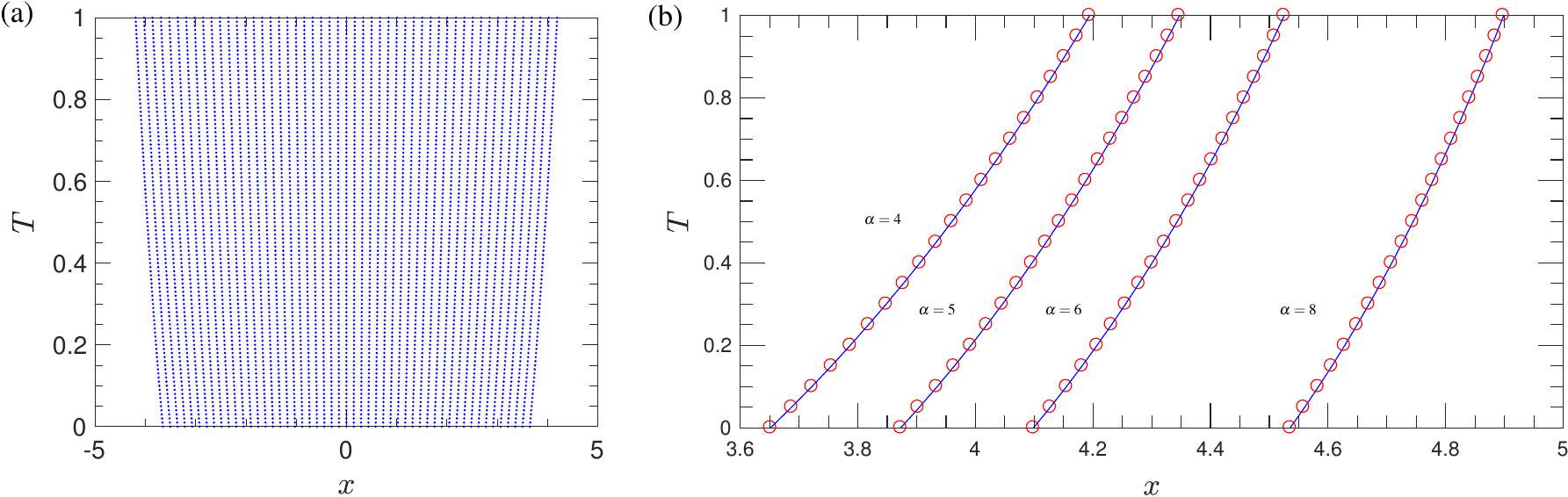}

    \caption{(a) The trajectory of the nodes for Barenblatt solution with $\alpha = 4$. (b) Movement of the interface for the Barenblatt solution: $\alpha = 4, 5, 6, 8$ [Exact: blue line; Numerical: red circle]. }\label{BP_Tra} 
\end{figure}

Quantitatively, we show the error of the right interface of our scheme 2 with different $\alpha$ ($\alpha = 4, 5, 6$ and $8$) at $T = 1$ in Table. \ref{Table2},  which shows that our scheme 2 can track the movement of the free boundary in second-order, even for large $\alpha$. 

\begin{table}[!bh]
  \begin{center}
         \begin{tabular}{ c | c | c | c }
    \hline
    {$\alpha$} &  $N = 51, \tau = 1/100$    &  $N = 101, \tau = 1/400$    & Order\\ \hline
          4    &  3.7241e-04  &  8.9717e-05 & 2.0534 \\ \hline
          5    &  4.4021e-04  &  1.0629e-04 & 2.0502   \\ \hline
          6    &  4.6867e-04   &  1.1329e-04 & 2.0486  \\ \hline
          8    &  4.8008e-04  &  1.1619e-04 & 2.0468   \\ \hline
         \end{tabular}         
  \caption{Numerical Error of right interface at $T = 1$ for initial data $\rho_0(X) = B_{\alpha}(X, 1)$ with different $\alpha$ ($\alpha = 4, 5, 6$ and $8$) at $T = 1$. The numerical solutions are computed by scheme 2.}
  \label{Table2}
\end{center}
\end{table}

\begin{remark}
  Although the scheme 2 is derived by the energy-dissipation law (\ref{New_Energy}), which is valid for $\alpha \geq 2$, numerical tests show that this scheme also works well for the case with $1 < \alpha < 2$.  Fig. \ref{scheme2_test} shows the numerical and exact solutions for $\alpha = 5/3$ and $\alpha = 1.1$ at $T = 1$ with $\rho_0(X) = B_{\alpha}(x, 1)$.
  \begin{figure}[!h]
    \includegraphics[width = \linewidth]{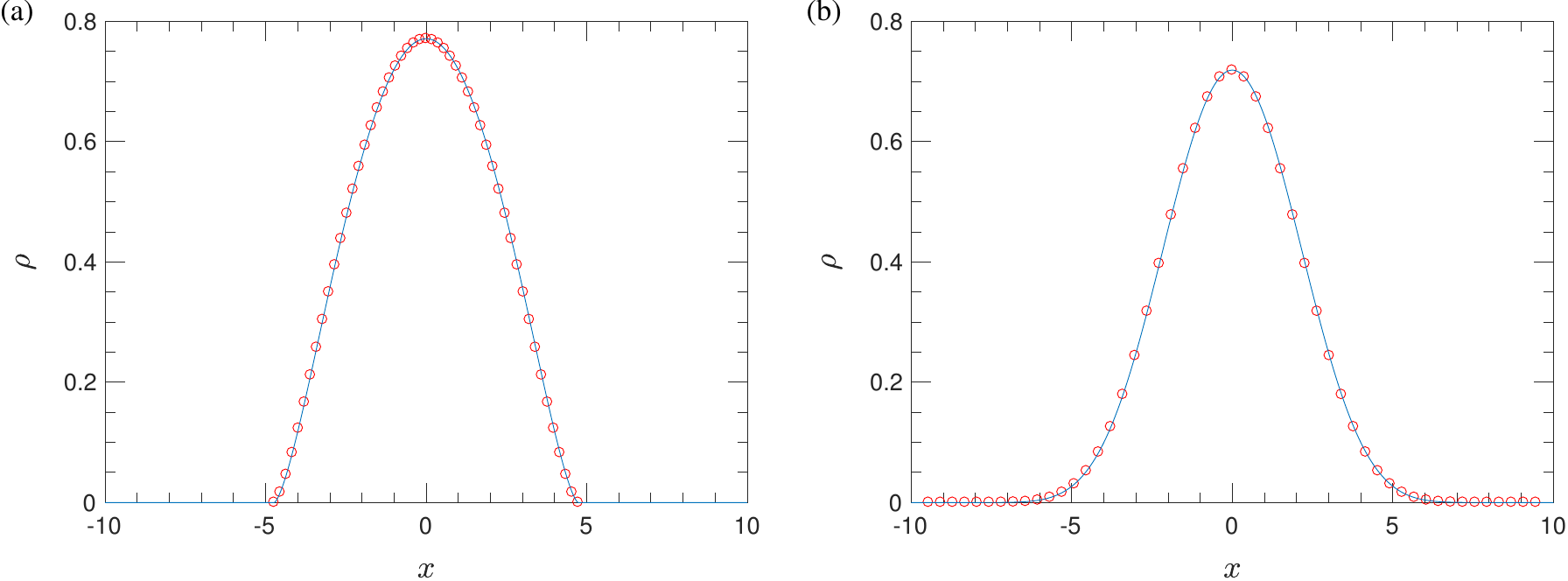}
    \caption{Solution for the PME with $\rho_0(X) = B_{\alpha}(x, 1)$ at $T = 1$ for (a) $\alpha = 5/3$, (b) $\alpha = 1.1$. The exact solutions are shown by solid line, while the numerical solutions computed by scheme 2 are shown by red-circle.}\label{scheme2_test} 
\end{figure}

\end{remark}

\subsubsection{Waiting Time}

 \begin{figure}[!b]
     \includegraphics[width = \linewidth]{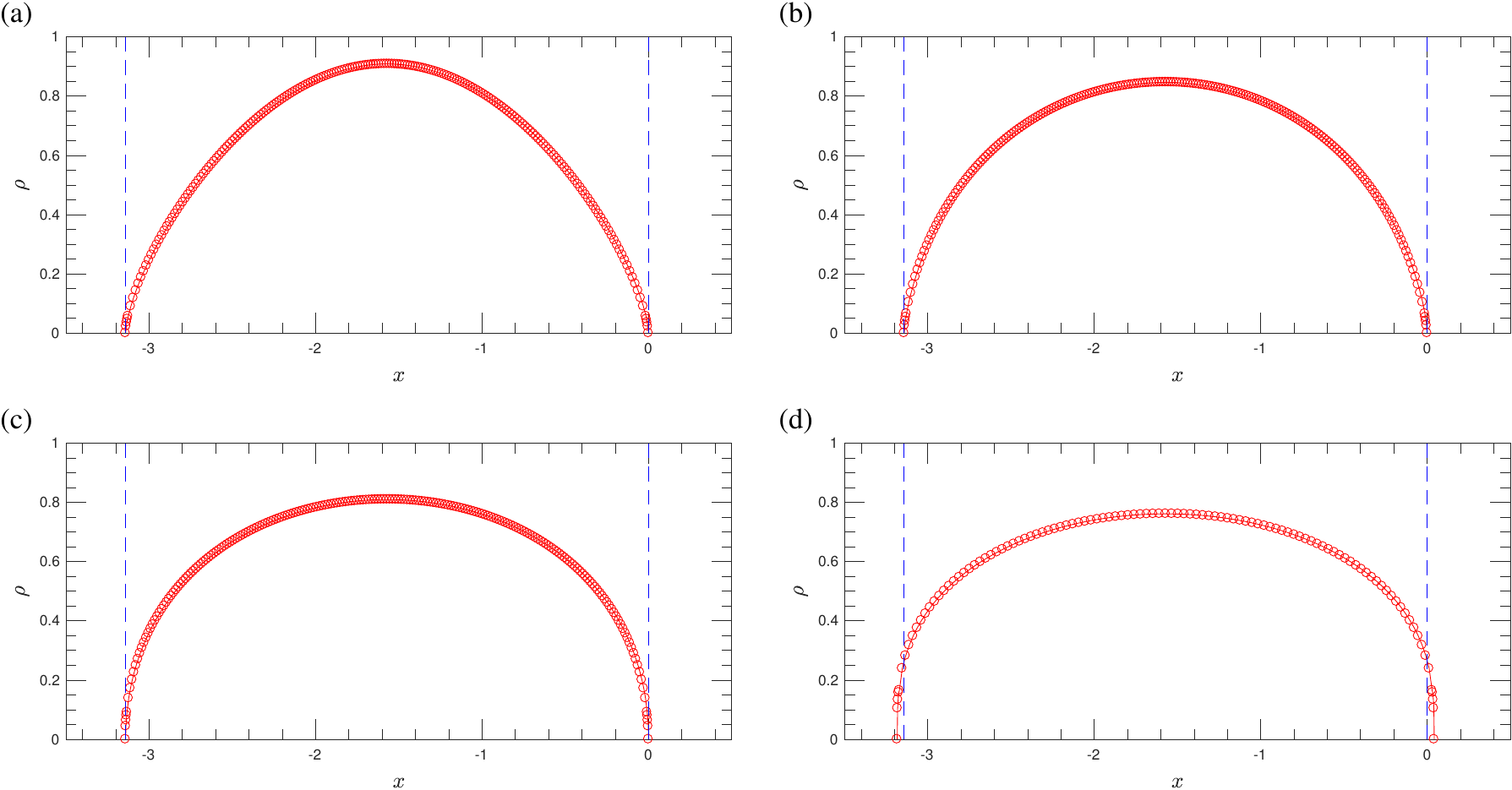}
    \caption{Numerical solution for the initial condition (\ref{ini_v0}) with $\theta = 0$ and $\alpha = 4$ at various time: (a) $t = 0$, (b) $t =0.05$, (c) $t = 0.1$, (d) $t = 0.2$,  in which the first and last element are manually refined ($N = 107$, $\tau = 10^{-4}$). }\label{WT_1D} 
\end{figure}
 Now, we study the waiting time phenomenon of the PME by our numerical methods. The waiting time phenomenon occurs for a certain type of initial data. For instance, if $\rho_0(X)$ satisfies
  \begin{equation}\label{ini_v0}
   \frac{\alpha}{\alpha-1} \rho_0^{\alpha-1}(X) = \begin{cases}
      &  (1 - \theta) \sin^2 X + \theta \sin^4 X , \quad \text{if}~ -\pi \leq X \leq 0, \\
      & 0, \quad \quad \quad \text{otherwise},
    \end{cases}
  \end{equation}
  then there exist a positive waiting time for $0 \leq \theta \leq 1$ \cite{aronson1983initially}. According the theoretical result \cite{aronson1983initially}, for $0 \leq \theta \leq 1/4$, the waiting time for the initial date (\ref{ini_v0}) is $\frac{1}{2 (\alpha + 1)(1 - \theta)}$.

Most of previous studies estimate the waiting time from the trajectory of numerical interface \cite{tomoeda1983numerical, bertsch1990numerical, ngo2018adaptive}, but the numerical waiting time may not be clearly estimated in some cases. 
As an alternative approach, \citeauthor{nakaki2003numerical} estimate the waiting time for the one-dimensional PME by transforming the original equation into another equation whose solution will blow up at the waiting time of the original PME \cite{nakaki2003numerical}. Recently, \citeauthor{duan2019numerical} proposed an elegant criterion to determine the waiting time in one-dimensional case, and they manually set the velocity of interface to be zero before the numerical waiting time. 
We can easily adopt the criterion proposed in \cite{duan2019numerical} into our numerical scheme in one-dimensional case. Fig. \ref{WT_1D} shows the numerical solutions for the initial data (\ref{ini_v0}) with $\theta = 0$ and $\alpha = 4$ at various time, in which we manually refine the first and last element to improve the numerical accuracy ($N = 107, \tau = 10^{-4}$). The numerical waiting time we obtained is $t = 0.1054$, which is consistent with theoretical results ($t = 0.1$).

\begin{figure}[!hb]
    \includegraphics[width = \linewidth]{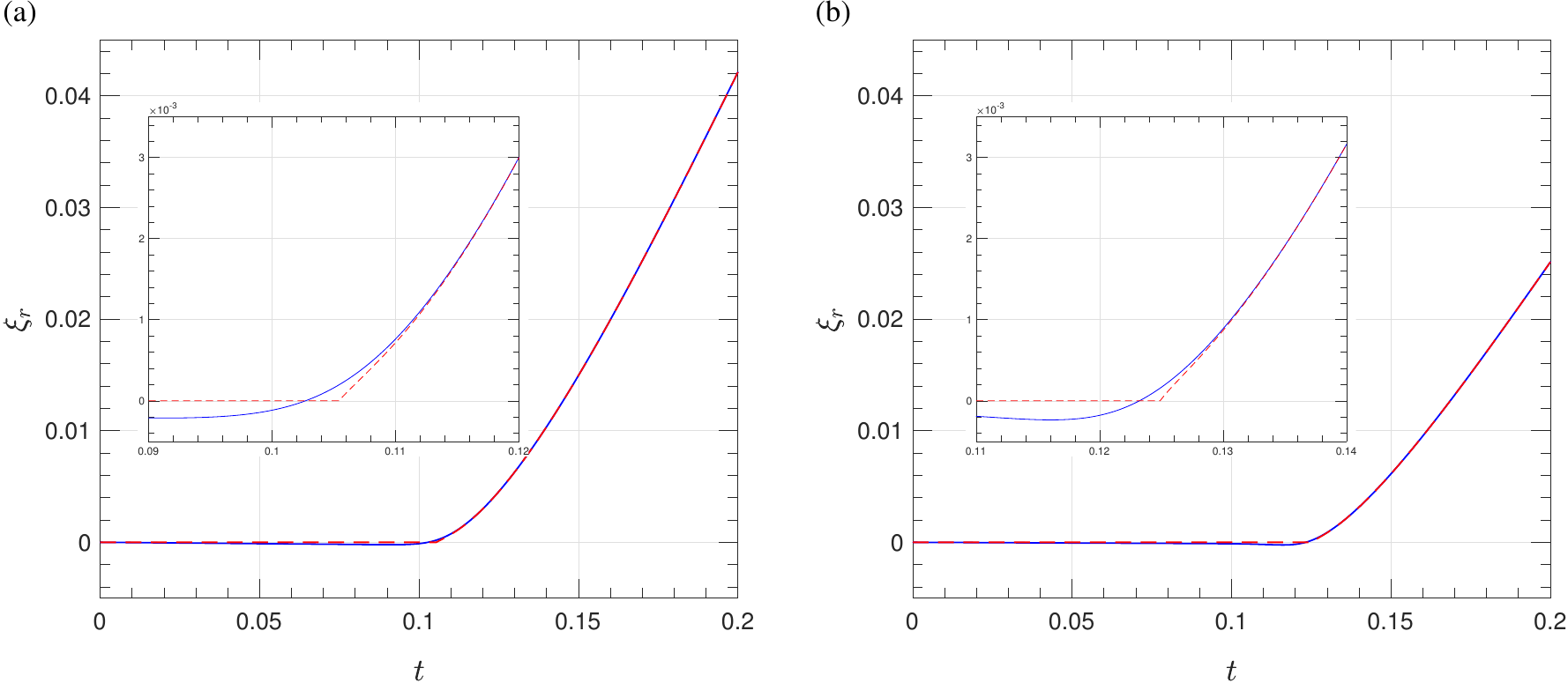}
    \caption{ (a) Numerical right interface computed by our numerical scheme with [red dashed-line] and without [blue solid-line] manually setting the velocity of interface to be zero before the numerical waiting time  for the initial data (\ref{ini_v0}) with $\theta = 0$ and $\alpha = 4$. (b) Numerical right interface computed by our numerical scheme with [red dashed-line] and without [blue solid-line] manually setting the velocity of interface to be zero before the numerical waiting time  for the initial data (\ref{ini_v0}) with $\theta = \frac{1}{2}$ and $\alpha = 7$. }\label{WT_1D_NI} 
\end{figure}

A natural question is whether we can estimate the waiting time from our numerical interface without manually setting the velocity of interface to be zero, as it is difficult to apply similar criterion to the PME in high spatial dimensions. Fig. \ref{WT_1D_NI} (a) shows the right numerical interface computed by our numerical method with and without manually setting the velocity of interface to be zero before the numerical waiting time for the initial data (\ref{ini_v0}) with $\theta = 0$ and $\alpha = 4$. Surprisingly, these two line almost coincide with each other after the numerical waiting time. Moreover, we noticed that the compact support will shrink at beginning due to the numerical approximation. These motivate us to define the numerical waiting time by the time when the numerical support begins to expand, i.e., 
\begin{equation}
t^{*}_{N} = \inf \left\{ t ~|~ |\xi_{b}(t)| > \xi_{b}^0, b = 1 ~\text{and}~ N   \right\},
\end{equation}
where $\xi_{b}(t)$ is the position of left and right interface at $t$ and $\xi_{b}^0$ is the initial position of left and right interface. Although no theoretical justification is available at this point, numerical experiments show that this criterion give a clear estimation to the waiting time for all cases that we tested. Fig. \ref{WT_1D_NI} (b) shows another example for the initial data (\ref{ini_v0}) with $\theta = \frac{1}{2}$ and $\alpha = 7$, in which the numerical interface for both approaches are shown. Both approaches indicate that the waiting time in this case is about $t^{*} \approx  0.124$.

\subsection{2D Simulations}

We now apply our numerical scheme 2, which is based on the energy-dissipation law (\ref{New_Energy}), to the PME in two dimensions. Although the explicit form of the numerical scheme cannot be given in a general mesh, using the explicit form of $F_e$ as a function of $a_{en(e, l)}$ and $b
_{en(e, l)}$ on each element $\tau_e$ (shown in the Appendix), we can compute $\frac{\delta \mathcal{A}_h}{\delta a_i}(\bm{a}^{n+1}, \bm{b}^{n+1})$ ($\frac{\delta \mathcal{A}_h}{\delta b_i}(\bm{a}^{n+1}, \bm{b}^{n+1})$) and ${\sf M}^{*}(\bm{a}^n, \bm{b}^n)$ by Algorithm \ref{Algorithm} and Algorithm \ref{calc_M} in remark 3.6, during the computer implementation.

\subsubsection{Barenblatt-Pattle solution}

We first validate our numerical scheme by studying the 2D Barenblatt-Pattle solution. The Barenblatt-Pattle solution in $d$-dimensions is given by
\begin{equation}\label{2D_B}
B_{\alpha} (\x ,t) = t^{-k} \left[ \left(C_0 - \frac{k(\alpha-1)}{2 d \alpha} \frac{|\x|^2}{t^{2k/d}}\right)_{+}  \right]^{1/(\alpha-1)},
\end{equation}
where $k = (\alpha - 1 + 2/d)^{-1}$, $C_0$ is a constant, related to the initial mass. This solution is radially symmetric, self-similar, and has compact support  $|\x| \leq \xi_{\alpha}(t)$ for any finite time, where
\begin{equation}
\xi_{\alpha} = \sqrt{\frac{2d \alpha C_0}{k (\alpha - 1)}} t^{k/d}.
\end{equation}

 \begin{figure}[!htb]
   \includegraphics[width = \linewidth]{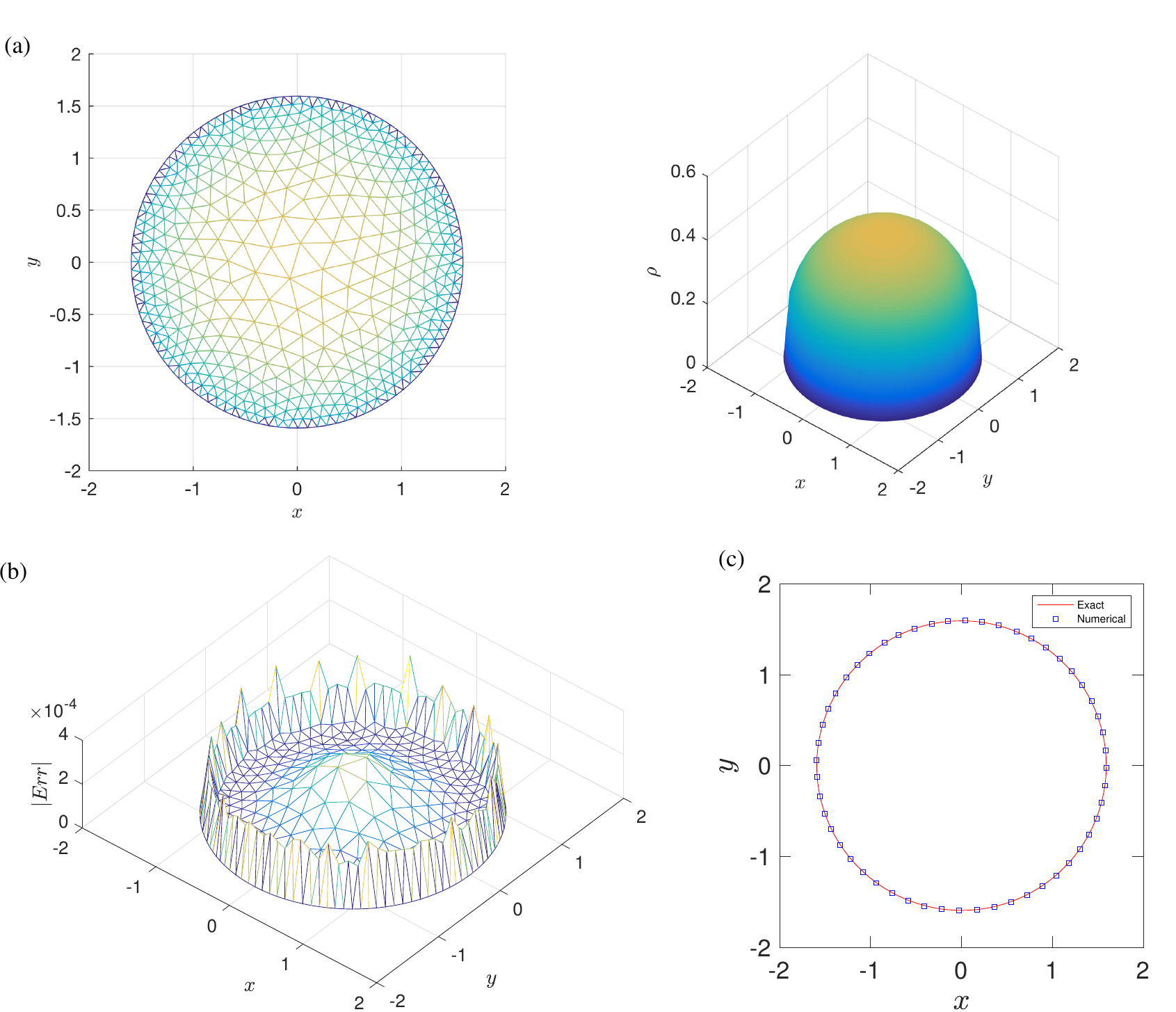}
    \caption{(a) Numerical solution at $T = 1$ (non-uniform mesh, $N = 516$) for the PME ($\alpha = 4$) with the initial data $B_4(\x, 1)$ ($C_0 = 0.1$) by scheme 2. (b) Numerical error for the numerical solution show in (a) at $T = 1$. (c) The numerical interface at $T = 1$ [shown by red circle] compared to the exact interface [blue line] for the PME ($\alpha = 4$) with the initial data $B_4(\x, 1)$ ($C_0 = 0.1$). }\label{PM_a_4} 
 \end{figure}

We take Barenblatt solution (\ref{2D_B}) for $C_0 = 0.1$ at $t = 1$ as the initial data, and compare the numerical solution at $T$ with the exact solution $B_{\alpha}(\x, T + 1)$.
Since the Barenblatt solution is steeper near the interface, we use the non-uniform mesh in $\Omega_0$ in our 2D simulation, which can largely reduce computational cost. The initial non-uniform mesh in $\Omega_0$ is generated by DistMesh \cite{persson2004simple}.

Fig. \ref{PM_a_4} shows the numerical solutions for $\alpha = 4$ at $T = 1$ in a non-uniform mesh with $N = 516$. The numerical error at $T = 1$ are plot in Fig. \ref{PM_a_4}(b), which indicates that the $L^{\infty}$-norm of the numerical solution is $O(10^{-4})$ for $N = 516$. The numerical and exact interfaces at $T = 1$ are shown in Fig. \ref{PM_a_4}(c), which demonstrated that our numerical scheme can track the movement of free boundary in 2D in high accuracy with a small $N$.


\begin{figure}[!htb]

  \includegraphics[width = \linewidth]{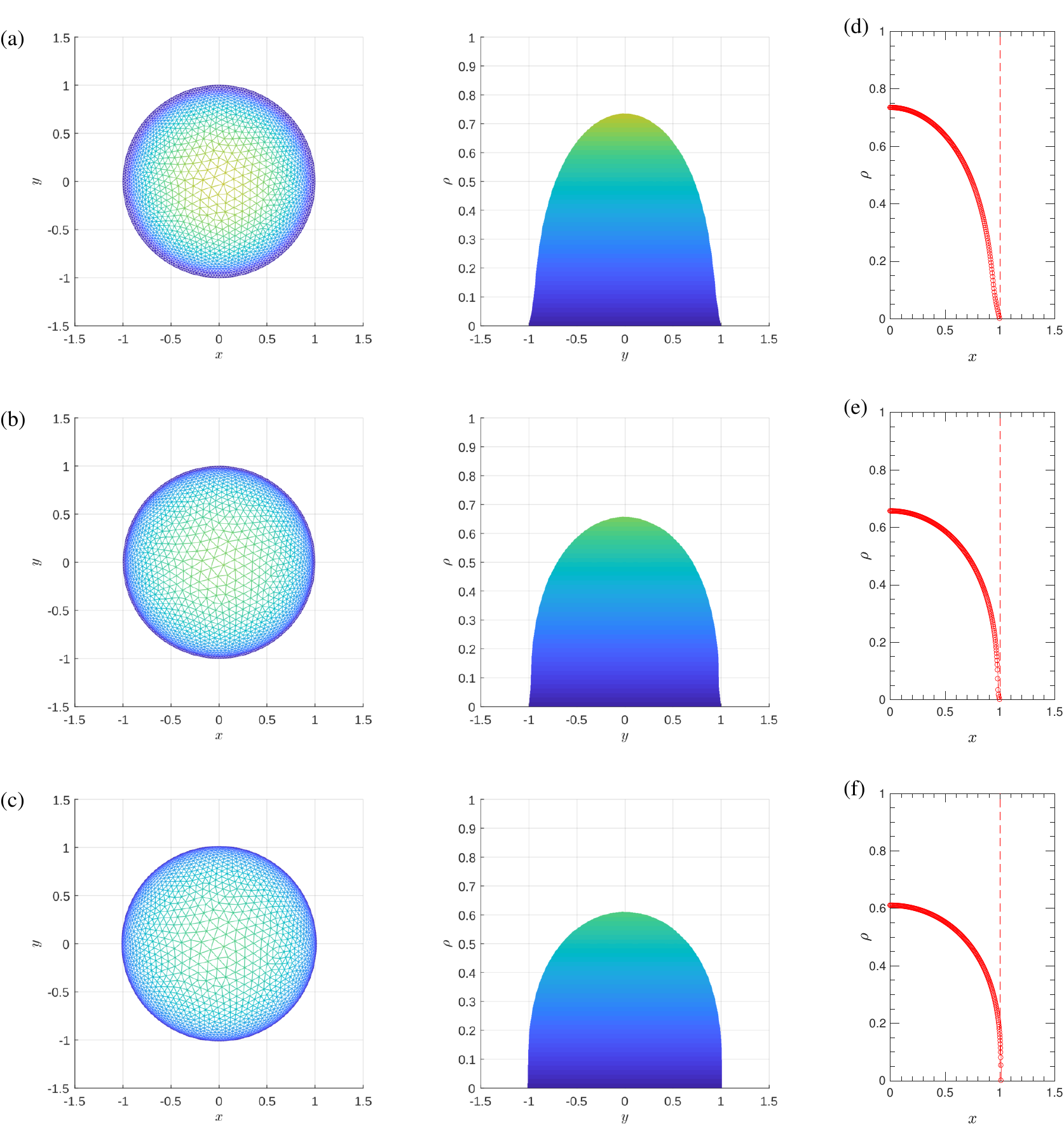}
       \caption{Numerical results for the initial data (\ref{WT_2D})  obtained by the 2D simulation [(a)-(c)] and 1D simulation with the axisymmetric assumption [(d)-(f)] at $t = 0.05, 0.1$ and $t = 0.15$.}\label{WT_2D}
\end{figure}

We now test the converge rate of 2D Barenblatt solution. The error for numerical solutions with $\rho_0(\X) = B_{\alpha}(\X, 1)$ at $T = 0.1$  in $\mathcal{L}^2$-norm for both $\alpha  = 2$ and $\alpha = 4$ is shown in Table \ref{Table3}. It shows that we can achieve a second-order convergence rate in space for both $\alpha  = 2$ and $\alpha = 4$ in $\mathcal{L}^2$-norm. 
\begin{table}[!h]
  \begin{center}
    {\footnotesize
  \begin{tabular}{ c  c | c  c | c  c |  c  c }
    \hline
    \multicolumn{4}{c|}{$\alpha = 2$}                         &  \multicolumn{4}{c}{$\alpha = 4$}  \\ \hline
    N      &  $\tau$  &  $\mathcal{L}^2$-error   & Order                &  N    &  $\tau$    & $\mathcal{L}^2$-error & Order     \\ \hline
    132    &  1/100   &  6.5388e-04    &         &  135     &  1/100      &  0.0066       &               \\ \hline
    524    &  1/400   &  1.6053e-04    & 2.0262  &  516  &  1/400      & 5.5299e-04    &  1.7807                \\  \hline  
    2103   &  1/1600  &  3.9867e-05    & 2.0096  & 2124  &  1/1600     & 1.6518e-04   &  1.9982          \\  \hline        
  \end{tabular}
  }
  \caption{The convergence rate of numerical solutions with $\rho_0(\X) = B_{\alpha}(\X, 1)$ at  $T = 0.1$ for the PME with $\alpha = 2$ and $\alpha = 4$.}
  \label{Table3}
\end{center}
\end{table}

\subsubsection{Waiting time}

Now, we apply our numerical scheme (scheme 2) to the PME with the initial data has a waiting-time phenomenon in two-dimensional situations. We take the initial value as:
\begin{equation}\label{WT_2D}
  \rho_0(X, Y) =
  \begin{cases}
    \cos ( \frac{\pi}{2} \sqrt{X^2 + Y^2} ), \quad & \text{for}~ \sqrt{X^2 + Y^2} \leq 1, \\
    0,  & \text{otherwise}, \\
  \end{cases}
\end{equation}
which has a positive waiting time. 

The numerical solutions for this initial data at various time are shown in Fig. \ref{WT_2D} [$N = 2105$ and $\tau = 10^{-3}$].  We use the non-uniform mesh in $\Omega_0$, which is dense around the free boundary. In order to validate our numerical results, we also compute the same initial date within the axisymmetric assumption, which can reduce the problem into a one-dimensional problem and enable us to apply the criterion in \cite{duan2019numerical} to estimate the waiting time. The numerical solutions obtained within the axisymmetric assumption for $N = 201$ and $\tau = 10^{-4}$ are shown in Fig. \ref{WT_2D} (d)-(f). 
\begin{figure}[!htb]

  \includegraphics[width = \linewidth]{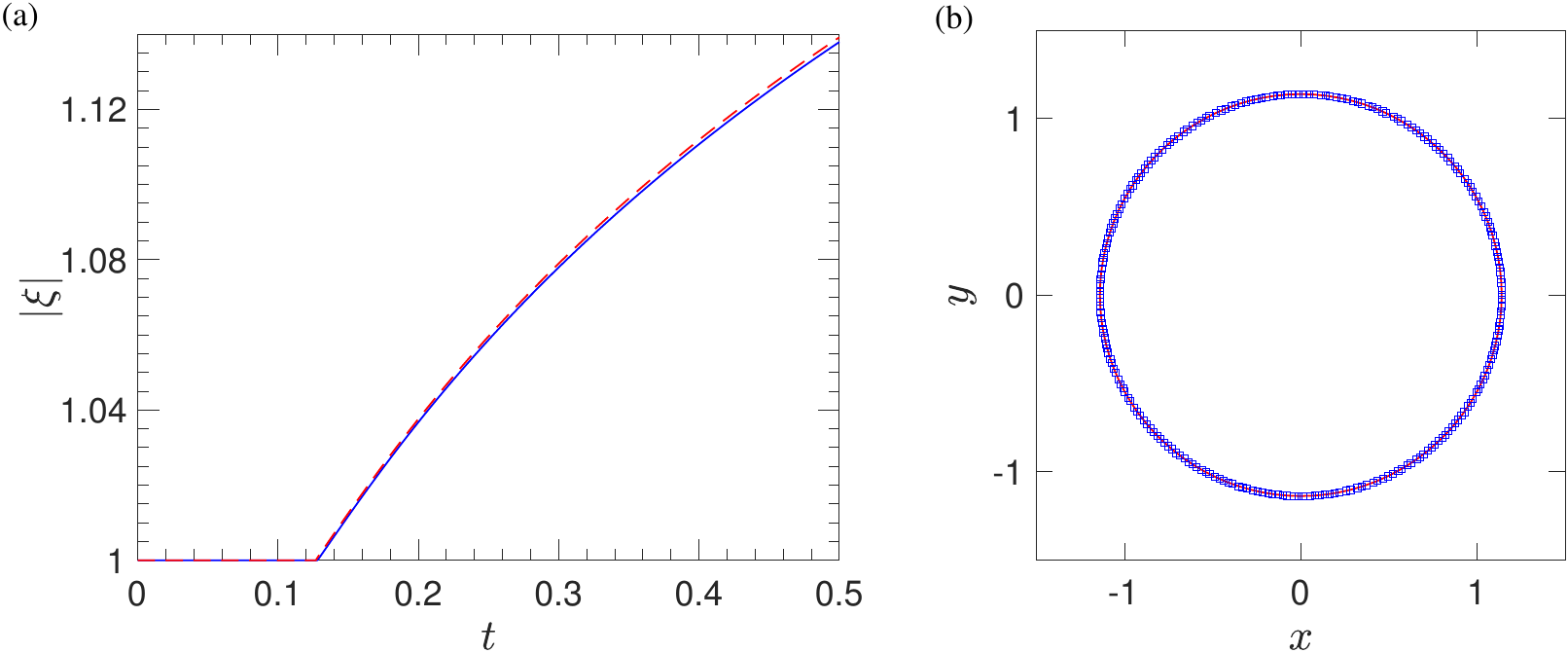}
      \caption{(a) The location of the free boundary obtained by the 2D simulation [blue solid line] and 1D simulation with the axisymmetric assumption [red dashed line]. (b) The free boundary at $t = 0.5$ by the 2D simulation [square] and 1D simulation with the axisymmetric assumption [solid line].  }\label{WT_NI_2D}
\end{figure}

The location of free boundary by the 2D simulation [blue solid line] and 1D simulation with the axisymmetric assumption [red dashed line] is plotted in Fig. \ref{WT_NI_2D}(a), where we define the location of free boundary in 2D simulation results by $$|\xi| = \max( \min \{ |\x_h(\X_b)| ~|~ \X_b \in \pp \Omega_0  \}, \xi_0),$$
where $\xi_0$ is the radius of the initial support. Both results indicate the wait time is around $0.128$. Fig. \ref{WT_NI_2D}(b) plot the  the numerical interface at $t = 0.5$ for both approaches, in which the free boundary in the 2D simulation result is shown by blue-square and the result obtained by 1D simulation is shown in red solid-line . It can be concluded that the numerical solutions obtained by both approaches are consistent, which validates our numerical scheme 2 in 2D, 
and we can have a clear estimation to the waiting time in 2D from the numerical interface obtained by the scheme 2. 

\subsubsection{Complex Support}

\begin{figure}[!htb]

  \centering
  \includegraphics[width = 0.9\linewidth]{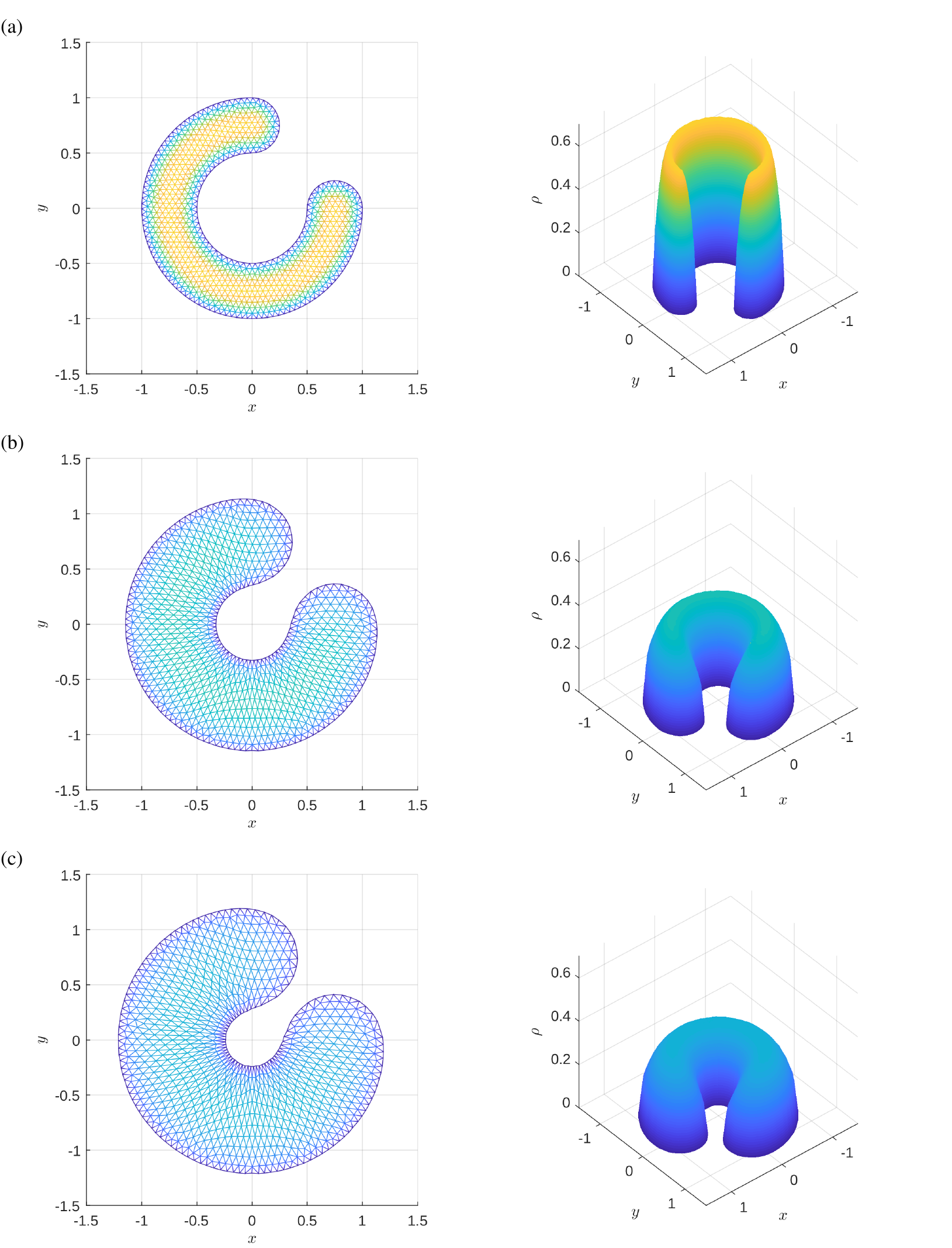}

  \caption{Numerical solutions of the PME with $\alpha = 3$ for the initial condition (\ref{DT_ini}) by scheme 2 at various time [$N = 910$, $\tau = 10^{-2}$]: (a) $t = 0$, (b) $t = 0.1$, (c) $t = 0.2$. }\label{DT}
\end{figure}

\begin{figure}[!htb]

  \centering
  \includegraphics[width = 0.9\linewidth]{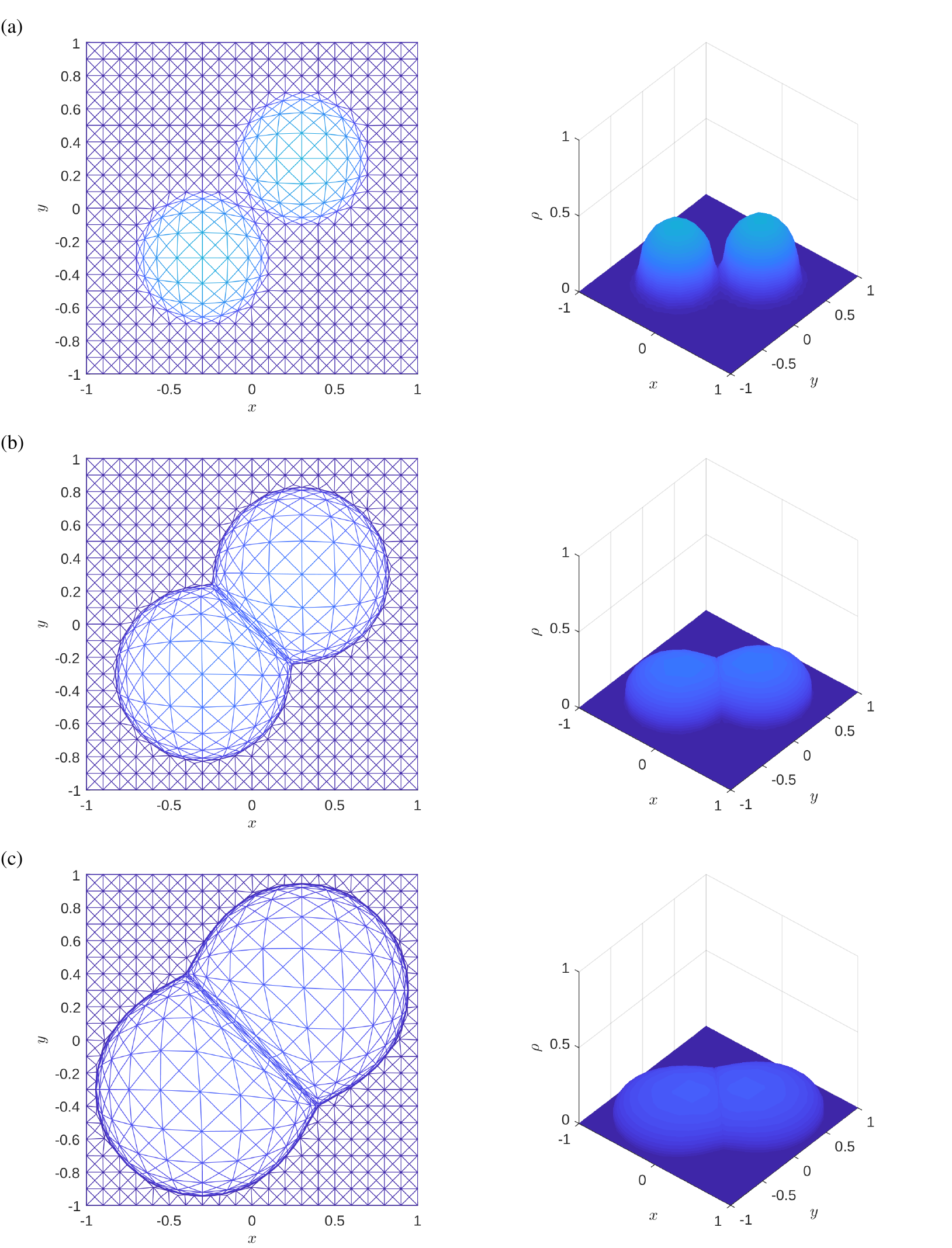}
     \caption{Numerical solutions for the PME with $\alpha = 4$ for the initial condition (\ref{DT_ini}) at various time by scheme 2: (a) $t = 0.1$, (b) $t = 1$, (c) $t = 5$. The initial mesh is a uniform mesh on $\Omega = [-1, 1]^2$ [$N =841$, $\tau = 10^{-2}$] .} \label{PM}

\end{figure}

Next we consider examples with complex compact supports in 2D. We first take the initial data as
\begin{equation}\label{DT_ini}
  \rho_0(\X)^{\alpha -1} =
  \begin{cases}
    & 25(0.25^2 - (\sqrt{X^2 + Y^2} - 0.75)^2)^{\frac{3}{2}}, \quad \sqrt{X^2 + Y^2} \in [0.5, 1]~\text{and}~(X < 0~\text{or}~Y < 0), \\
    & 25(0.25^2  - X^2 - (Y - 0.75)^2)^{\frac{3}{2}}, \quad X^2 + (Y - 0.75)^2 \leq 0.25^2 ~\text{and}~X \geq 0, \\
    & 25(0.25^2  - (X - 0.75)^2 - Y^2)^{\frac{3}{2}}, \quad (X - 0.75)^2 + Y^2 \leq 0.25^2 ~\text{and}~Y \geq 0, \\
    & 0, \quad \text{otherwise}, \\
  \end{cases}
\end{equation}
which has a partial donut-shaped support. Similar examples are studied in \cite{baines2005moving, ngo2018adaptive}.

Fig. \ref{DT} shows the numerical solutions of the PME with $\alpha = 3$ for the initial condition (\ref{DT_ini}) at various time, which are computed by scheme 2 with $N = 910$ and $\tau = 10^{-2}$. One can see that our method works well for the concave domain. However, as pointed out in \cite{baines2005moving}, a Lagrangian method can not handle the topological change automatically, which is also a limitation of our numerical approach. For this example, since the domain will evolve to reach a point where two ends of the ``horseshoe'' intersecting each other, the tangling of mesh cannot be avoided after this point.
In order to get solutions beyond this point, one can manually interpolate the solution on to a new mesh, which can be viewed as an update to $\rho_0(\X)$ (along with the mesh).

In our last numerical example, we study a peaks merge problem for the PME with $\alpha = 4$, in which the initial data has two peaks, connected by a thin layer of mass. Similar numerical experiments are studied in \cite{ngo2017study, carrillo2018lagrangian}.

Let $\Omega = [-1, 1]^2$, consider the initial data
\begin{equation}
\rho_0(X, Y) = e^{-20 ((X-0.3)^2 + (Y - 0.3)^2)} + e^{-20((X + 0.3)^2 + (Y + 0.3)^2)} + 0.001.
\end{equation}
The boundary condition on $\pp \Omega$ is the Neumann boundary condition. In our numerical simulation, we take $\Omega_0 = \Omega$ and manually set the velocity of the nodes on $\pp \Omega_0$ to be zero for the Neumann boundary condition. Fig. \ref{PM} shows the numerical solutions by scheme 2 at various time ($t = 0.1, 1$ and $5$) with a uniform initial mesh in $\Omega_0$ [$N =841$, $\tau = 10^{-2}$]. It can be seen that the mesh is concentrated around the ``interface'', in which the solution is steep, during the evolution. One can view our approach as a kind of moving mesh method. Unlike most of traditional moving mesh approach, in which both updating the mesh and solving the equation in the new mesh are required, we only need to solve the equation of the flow map, i.e., the equation of the mesh, the numerical solutions are determined by the kinematic relation.   
Similar to \cite{carrillo2018lagrangian}, it is a surprise that our numerical method can handle the situation of ``peaks merge'', although local coarsening and remeshing is still needed in order to get better numerical solutions.



\section{Summary}

Structure-preserving and adaptive are two important aspects in designing efficient numerical methods for PDEs arising in numerous physical and biological modeling.
 It is usually a difficult task to construct a numerical scheme to a conservative or dissipative PDE that retain the conservation/dissipation properties in a discrete sense \cite{furihata2010discrete}. 
 In this paper, we proposed a general framework to derive an efficient structure-preserving numerical scheme by employing a discrete energetic variational approach.
A discrete energetic variational approach provides basis of deriving the ``semi-discrete equations'' after introducing a proper spatial discretization to the given energy-dissipation law, and can be applied to a large class of partial differential equations with energy-dissipation laws and kinematic relations, such as generalized diffusion equations, phase-field equations, and equations of liquid crystal.
Within a piecewise linear approximation to the flow map, our approach is capable of handling high spatial dimensional situations.
 As an application, we develop two numerical schemes for the PME based on different energy-dissipation laws. 
 By performing numerical experiments in both 1D and 2D, we show that the numerical scheme based on the energy-dissipation law (\ref{New_Energy}) can better capture the free boundary and estimate the waiting time for the PME, without explicitly tracking the movement of the free boundary. Although our Lagrangian scheme performs well in numerical tests, a detailed numerical analysis of such schemes are required. We are working hard in this direction. As mentioned in the sect. 3, there is essential difference between Lagrangian schemes in 1D and high dimensions since both $\det F$ and the admissible set are not convex, which impose challenges in numerical analysis.





On the notion of our numerical approach, large deformations and topological changes can be difficult to handle by the dynamics of the flow map, usually with higher nonlinearity, degeneracy, or possible singularities. Furthermore, for general problems, a proper ``initial data'' may not be available. In order to solve these problems, we will employ a hybrid method, which solves the original equation in Eulerian and Lagrangian coordinates alternatively, in the ongoing work. In other word, during the evolution of the flow map, we can update the ``initial data'' (``reference data'') and the mesh, where the ``initial data'' is obtained by some Eulerian methods.

Finally, in the current approach, our piecewise linear approximation to the flow map is based on a finite element method, which provides us a simple framework to compute the deformation matrix $F$ on each element explicitly, but it is not obvious to incorporate the local coarsening and remeshing into it. One possible idea is to incorporate the particle methods \cite{chertock2017practical} into our framework, in which remeshing for particle distortion can be easily dealt with, although how to compute the deformation matrix $F$ might still be a challenge.


\section*{Acknowledgment}
The authors acknowledge the partial support of NSF (Grant DMS-1759536). We thank Prof. Cheng Wang, Prof. Xingye Yue, Dr. Chenghua Duan, Dr. Pei Liu and Dr. Qing Cheng for helpful discussions.  Y. Wang would also like to thank Department of Applied Mathematics at Illinois Institute of Technology for their generous support and for a stimulating environment.



\appendix

\section{Explicit form of $F_e$ in 2D}
In the appendix, we give the explicit form of $F_e$ as a function of  $a_{en(e, l)}$ and $b_{en(e, l)}$ ($l = 1, 2, 3$) on each element $\tau_e$.

For a given element $\tau_e$, we denote the nodes of it by $\X^e_l$, and we let $\avec^e_l = (a_l^e, b_l^e) \triangleq \avec_{en(l, e)}$ ($l = 1, 2, 3$). The discrete flow map on $\tau_e$ can be written as
\begin{equation}\label{nodal_tau_e}
  \x_h(\X) = \sum_{l = 1}^{3} \avec_l^e \lambda_{l}^e(\X), \quad \X \in \tau_e,
\end{equation}
where $\lambda^e_{l} \in P_1(\tau_e)$ is a nodal basis on $\tau_e$, which satisfies $\lambda^e_{l}(\X^e_{m}) = \delta_{lm}$ ($l, m  = 1, 2 , 3$).

We can compute $\lambda_{l}^e(\X)$ by mapping $\tau_e$ into the reference triangle $\tau_{s} = \{ \hat{\X} = (\hat{X}, \hat{Y}) \in \mathbb{R}^2 ~:~ \hat{X} \geq 0, \hat{X} + \hat{Y} \leq 1  \}$ with nodes $\hat{\X}_1^s = (0, 0), \hat{\X}_2^s = (1, 0)$, and $\hat{\X}_3^s = (0, 1)$. The map between $\tau_{s}$ to $\tau_e$ is given by
\begin{equation}
\X = L_e(\hat{\X}) = A_e \hat{\X} + b_e,
\end{equation}
where
\begin{equation}
  A_e = 
  \begin{pmatrix}
    X^e_2 - X^e_1 &  X^e_3 - X^e_1 \\
    Y^e_2 - Y^e_1 &  Y^e_3 - Y^e_1 \\
  \end{pmatrix},
   \quad b_e = \X^e_1.
\end{equation}
Hence,
\begin{equation}
\lambda_{l}^e(\X) = \lambda^{s}_{l} (L_e^{-1} (\X)), \quad l = 1, 2, 3.
\end{equation}

Since the nodal basis on $\tau_s$ is given by
\begin{equation}
\lambda_1^s(\hat{X}, \hat{Y}) = 1 - \hat{X} - \hat{Y}, \quad \lambda_2^s(\hat{X}, \hat{Y}) = \hat{X}, \quad \lambda_3^s(\hat{X}, \hat{Y}) = \hat{Y},
\end{equation}
we have
\begin{equation}
\nabla_{\X} \lambda^e_{l}(\X) =  A_e^{-\rm{T}} \nabla_{\xi}\lambda^{s}_{l}(\hat{\X}),
\end{equation}
where
\begin{equation*}
  \nabla_{\xi} \lambda^{s}_1 = 
   \begin{pmatrix}
      -1 \\
     -1 \\
   \end{pmatrix},
   \quad
     \nabla_{\xi} \lambda^{s}_2 = 
   \begin{pmatrix}
     1 \\
     0 \\
   \end{pmatrix},
   \quad
     \nabla_{\xi} \lambda^{s}_3 = 
   \begin{pmatrix}
     0 \\
     1 \\
   \end{pmatrix},
   \quad
   A_e^{-1} = \frac{1}{\det A_e}
      \begin{pmatrix}
     & Y_3^e - Y_1^e & X_1^e - X_3^e \\
     & Y_1^e - Y_2^e & X_2^e - X_1^e \\
   \end{pmatrix}.
\end{equation*}

Then, we can compute $F_e$ as a function of $a_l^e$ and $b_l^e$ on $\tau_e$ directly by (\ref{nodal_tau_e}), which is 
\begin{equation}
  \begin{aligned}
    F_e (a_l^e, b_l^e) 
   &  =  \sum_{l = 1}^{3} \avec_{en(l, e)} \otimes \nabla_{X} \lambda_{\alpha}^e(\X) \\  
 &  =    \begin{pmatrix}
     a_e^1 &  a_e^2 &  a_e^3\\
     b_e^1 &  b_e^2 & b_e^3\\
    \end{pmatrix} \nabla_{\xi} \bm{\lambda}^{s}(\hat{\X})A_e^{-1} \\
    & = \frac{1}{\det A_e} \begin{pmatrix}
     a_e^1 Y_{2-3}^e +  a_e^2Y_{3-1}^e +  a_e^3Y_{1-2}^e &  a_e^1  X_{3-2}^e +  a_e^2 X_{1-3}^e +  a_e^3 X_{2-1}^e  \\
     b_e^1 Y_{2-3}^e +  b_e^2Y_{3-1}^e +  b_e^3 Y_{1-2}^e  & b_e^1 X_{3-2}^e +  b_e^2 X_{1-3}^e +  b_e^3 X_{2-1}^e  \\
    \end{pmatrix},
  \end{aligned}
\end{equation}
where $X_{l - m}^e = X_{\l}^e - X_{m}^e$, $Y_{l - m}^e = Y_{l}^e - Y_{m}^e$ ($l, m  = 1, 2, 3$).

\section{Numerical Implementation in 2D}
In this appendix, we present the detailed numerical implementation for the numerical scheme (\ref{Scheme_2D}), which cannot be written down in an explicit form in a general mesh as $\frac{\delta \mathcal{A}}{\delta a_i}({\bm a}^{n+1}, {\bm b}^{n+1})$ ($\frac{\delta \mathcal{A}}{\delta b_i}({\bm a}^{n+1}, {\bm b}^{n+1})$) and ${\sf M}^*({\bm a}^n, {\bm b}^n)$ depend on the triangulation $\mathcal{T}_h$. Since we have the explicit form of $F_e$ as a function of $a_{en(e, l)}$ and $b_{en(e, l)}$ on
  each element $\tau_e$ (shown in the Appendix A), we can compute $\frac{\delta \mathcal{A}}{\delta a_i}({\bm a}^{n+1}, {\bm b}^{n+1})$ ($\frac{\delta \mathcal{A}}{\delta b_i}({\bm a}^{n+1}, {\bm b}^{n+1})$) and ${\sf M}^*({\bm a}^n, {\bm b}^n)$ in the numerical implementation by using the standard technique in finite element methods, that is, summing the results on each element over the mesh \cite{larson2013finite}. For instance,  $\frac{\delta \mathcal{A}_h}{\delta a_i}(\bm{a}^{n+1}, \bm{b}^{n+1})$ ($\frac{\delta \mathcal{A}_h}{\delta b_i}(\bm{a}^{n+1}, \bm{b}^{n+1})$) and ${\sf M}^{*}(\bm{a}^n, \bm{b}^n)$ in our scheme 2 [scheme (\ref{Scheme_2D}) with (\ref{PME_scheme_2})] can be computed by Algorithm \ref{Algorithm} and Algorithm \ref{calc_M}. 
  We can have an explicit form of $\frac{\delta \mathcal{A}_h}{\delta a_i}(\bm{a}^{n+1}, \bm{b}^{n+1})$ ($\frac{\delta \mathcal{A}_h}{\delta b_i}(\bm{a}^{n+1}, \bm{b}^{n+1})$) and ${\sf M}^{*}(\bm{a}^n, \bm{b}^n)$  in a uniform triangulation in a rectangular domain by applying Algorithm \ref{Algorithm} and Algorithm \ref{calc_M}.

\begin{algorithm}[!h]
  {\footnotesize
    \caption{ {\bf Assembly of $\frac{\delta \mathcal{A}_h}{\delta a_i} (\bm{a}^{n+1}, \bm{b}^{n+1})$ or $\frac{\delta \mathcal{A}_h}{\delta b_i} (\bm{a}^{n+1}, \bm{b}^{n+1})$}} \label{Algorithm}
    \SetKwData{Left}{left}\SetKwData{This}{this}\SetKwData{Up}{up}
    \SetKwFunction{Union}{Union}\SetKwFunction{FindCompress}{FindCompress}
    
      \For{$e = 1, 2, \ldots M$}{
      Compute $\frac{\delta \mathcal{A}_h^e}{\delta a_{en(e, l)}}(\bm{a}^{n+1}, \bm{b}^{n+1})$ and $\frac{\delta \mathcal{A}_h^e}{\delta b_{en(e, l)}} (\bm{a}^{n+1}, \bm{b}^{n+1})$ on $\tau_e$ by (\ref{d_action_el}) for $l = 1, 2, 3$.

      \For{$l = 1, 2, 3$}{
        %
              
        \begin{equation*}
          \begin{aligned}
          & \frac{\delta \mathcal{A}_h}{\delta a_{en(e, l)}} (\bm{a}^{n+1}, \bm{b}^{n+1}) = \frac{\delta \mathcal{A}_h}{\delta a_{en(e, l)}} (\bm{a}^{n+1}, \bm{b}^{n+1})  + \frac{\delta \mathcal{A}_h^e}{\delta a_{en(e, l)}}(\bm{a}^{n+1}, \bm{b}^{n+1}) \\
              &  \frac{\delta \mathcal{A}_h}{\delta b_{en(e,l)}} (\bm{a}^{n+1}, \bm{b}^{n+1}) = \frac{\delta \mathcal{A}_h}{\delta b_{en(e, l)}} (\bm{a}^{n+1}, \bm{b}^{n+1})  + \frac{\delta \mathcal{A}_h^e}{\delta b_{en(e, l)}}(\bm{a}^{n+1}, \bm{b}^{n+1}) \\
        \end{aligned}
        \end{equation*}
      }

      }
      }
\end{algorithm}

\begin{algorithm}[!h]
  {\footnotesize
    \caption{{\bf Assembly of ${\sf M}^{*}({\bm a}^n, {\bm b}^n)$}} \label{calc_M}
    \SetKwData{Left}{left}\SetKwData{This}{this}\SetKwData{Up}{up}
    \SetKwFunction{Union}{Union}\SetKwFunction{FindCompress}{FindCompress}
      \For{$e = 1, 2, \ldots M$}{
        On $\tau_e$, compute $M^e({\bm a}^n, {\bm b}^n) = \frac{1}{12}
          \begin{pmatrix}
            2 & 1 & 1 \\
            1 & 2 & 1 \\
            1 & 1 & 2 \\
          \end{pmatrix} |\tau_e| \det F_e^n$.
    
      
      \For{$l = 1, 2, 3$}{
        \For{$m = 1, 2, 3$}{
          $$M_{en(e,l) en(e, m)}^n({\bm a}^n, {\bm b}^n) = M_{en(e,l) en(e, m)}({\bm a}^n, {\bm b}^n) + M^e_{lm}({\bm a}^n, {\bm b}^n)$$
        }
      }
      }
      }
\end{algorithm}

\bibliographystyle{plainnat}
\bibliography{PME}

\begin{thebibliography}{65}
\providecommand{\natexlab}[1]{#1}
\providecommand{\url}[1]{\texttt{#1}}
\expandafter\ifx\csname urlstyle\endcsname\relax
  \providecommand{\doi}[1]{doi: #1}\else
  \providecommand{\doi}{doi: \begingroup \urlstyle{rm}\Url}\fi

\bibitem[Alikakos(1985)]{alikakos1985pointwise}
Nicholas~D Alikakos.
\newblock On the pointwise behavior of the solutions of the porous medium
  equation as t approaches zero or infinity.
\newblock \emph{Nonlinear Analysis: Theory, Methods \& Applications},
  9\penalty0 (10):\penalty0 1095--1113, 1985.

\bibitem[Arnol'd(2013)]{arnol2013mathematical}
Vladimir~Igorevich Arnol'd.
\newblock \emph{Mathematical methods of classical mechanics}, volume~60.
\newblock Springer Science \& Business Media, 2013.

\bibitem[Aronson et~al.(1983)Aronson, Caffarelli, and
  Kamin]{aronson1983initially}
DG~Aronson, LA~Caffarelli, and S~Kamin.
\newblock How an initially stationary interface begins to move in porous medium
  flow.
\newblock \emph{SIAM Journal on Mathematical Analysis}, 14\penalty0
  (4):\penalty0 639--658, 1983.

\bibitem[Baines et~al.(2006)Baines, Hubbard, Jimack, and
  Jones]{baines2006scale}
M~J Baines, M~E Hubbard, P~K Jimack, and A~C Jones.
\newblock Scale-invariant moving finite elements for nonlinear partial
  differential equations in two dimensions.
\newblock \emph{Applied Numerical Mathematics}, 56\penalty0 (2):\penalty0
  230--252, 2006.

\bibitem[Baines et~al.(2005)Baines, Hubbard, and Jimack]{baines2005moving}
Mike~J Baines, ME~Hubbard, and PK~Jimack.
\newblock A moving mesh finite element algorithm for the adaptive solution of
  time-dependent partial differential equations with moving boundaries.
\newblock \emph{Applied Numerical Mathematics}, 54\penalty0 (3-4):\penalty0
  450--469, 2005.

\bibitem[Bertsch and Dal~Passo(1990)]{bertsch1990numerical}
M~Bertsch and R~Dal~Passo.
\newblock A numerical treatment of a superdegenerate equation with applications
  to the porous media equation.
\newblock \emph{Quarterly of applied mathematics}, 48\penalty0 (1):\penalty0
  133--152, 1990.

\bibitem[Budd et~al.(1999)Budd, Collins, Huang, and Russell]{budd1999self}
C~J Budd, G~J Collins, W~Z Huang, and R~D Russell.
\newblock Self--similar numerical solutions of the porous--medium equation
  using moving mesh methods.
\newblock \emph{Philosophical Transactions of the Royal Society of London.
  Series A: Mathematical, Physical and Engineering Sciences}, 357\penalty0
  (1754):\penalty0 1047--1077, 1999.

\bibitem[Cahn and Hilliard(1958)]{cahn1958free}
John~W Cahn and John~E Hilliard.
\newblock Free energy of a nonuniform system. i. interfacial free energy.
\newblock \emph{The Journal of chemical physics}, 28\penalty0 (2):\penalty0
  258--267, 1958.

\bibitem[Carrillo and Moll(2009)]{carrillo2009numerical}
Jos{\'e}~A Carrillo and J~Salvador Moll.
\newblock Numerical simulation of diffusive and aggregation phenomena in
  nonlinear continuity equations by evolving diffeomorphisms.
\newblock \emph{SIAM Journal on Scientific Computing}, 31\penalty0
  (6):\penalty0 4305--4329, 2009.

\bibitem[Carrillo et~al.(2016)Carrillo, Ranetbauer, and
  Wolfram]{carrillo2016numerical}
Jos{\'e}~A Carrillo, Helene Ranetbauer, and Marie-Therese Wolfram.
\newblock Numerical simulation of nonlinear continuity equations by evolving
  diffeomorphisms.
\newblock \emph{Journal of Computational Physics}, 327:\penalty0 186--202,
  2016.

\bibitem[Carrillo et~al.(2018)Carrillo, D{\"u}ring, Matthes, and
  McCormick]{carrillo2018lagrangian}
Jos{\'e}~A Carrillo, Bertram D{\"u}ring, Daniel Matthes, and David~S McCormick.
\newblock A lagrangian scheme for the solution of nonlinear diffusion equations
  using moving simplex meshes.
\newblock \emph{Journal of Scientific Computing}, 75\penalty0 (3):\penalty0
  1463--1499, 2018.

\bibitem[Carrillo et~al.(2017)Carrillo, Huang, Patacchini, and
  Wolansky]{carrillo2017numerical}
Jos{\'e}~Antonio Carrillo, Yanghong Huang, Francesco~Saverio Patacchini, and
  Gershon Wolansky.
\newblock Numerical study of a particle method for gradient flows.
\newblock \emph{Kinetic \& Related Models}, 10\penalty0 (3):\penalty0 613--641,
  2017.

\bibitem[Cavalli et~al.(2007)Cavalli, Naldi, Puppo, and
  Semplice]{cavalli2007high}
Fausto Cavalli, Giovanni Naldi, Gabriella Puppo, and Matteo Semplice.
\newblock High-order relaxation schemes for nonlinear degenerate diffusion
  problems.
\newblock \emph{SIAM Journal on Numerical Analysis}, 45\penalty0 (5):\penalty0
  2098--2119, 2007.

\bibitem[Chertock(2017)]{chertock2017practical}
Alina Chertock.
\newblock A practical guide to deterministic particle methods.
\newblock In \emph{Handbook of Numerical Analysis}, volume~18, pages 177--202.
  Elsevier, 2017.

\bibitem[Christiansen et~al.(2011)Christiansen, Munthe-Kaas, and
  Owren]{christiansen2011topics}
Snorre~H Christiansen, Hans~Z Munthe-Kaas, and Brynjulf Owren.
\newblock Topics in structure-preserving discretization.
\newblock \emph{Acta Numerica}, 20:\penalty0 1--119, 2011.

\bibitem[Davis and Gartland~Jr(1998)]{davis1998finite}
Timothy~A Davis and Eugene~C Gartland~Jr.
\newblock Finite element analysis of the landau--de gennes minimization problem
  for liquid crystals.
\newblock \emph{SIAM Journal on Numerical Analysis}, 35\penalty0 (1):\penalty0
  336--362, 1998.

\bibitem[DiBenedetto and Hoff(1984)]{dibenedetto1984interface}
E~DiBenedetto and David Hoff.
\newblock An interface tracking algorithm for the porous medium equation.
\newblock \emph{Transactions of the American Mathematical Society},
  284\penalty0 (2):\penalty0 463--500, 1984.

\bibitem[Duan et~al.(2019)Duan, Liu, Wang, and Yue]{duan2019numerical}
Chenghua Duan, Chun Liu, Cheng Wang, and Xingye Yue.
\newblock Numerical methods for porous medium equation by an energetic
  variational approach.
\newblock \emph{Journal of Computational Physics}, 2019.

\bibitem[Eisenberg et~al.(2010)Eisenberg, Hyon, and Liu]{eisenberg2010energy}
Bob Eisenberg, Yunkyong Hyon, and Chun Liu.
\newblock Energy variational analysis of ions in water and channels: Field
  theory for primitive models of complex ionic fluids.
\newblock \emph{The Journal of Chemical Physics}, 133\penalty0 (10):\penalty0
  104104, 2010.

\bibitem[Furihata and Matsuo(2010)]{furihata2010discrete}
Daisuke Furihata and Takayasu Matsuo.
\newblock \emph{Discrete variational derivative method: a structure-preserving
  numerical method for partial differential equations}.
\newblock Chapman and Hall/CRC, 2010.

\bibitem[G(1952)]{gi1952some}
Barenblatt G, I.
\newblock On some unsteady motions of a liquid and gas in a porus medium.
\newblock \emph{Prikl. Mat. Mekh.}, 16:\penalty0 67--78, 1952.

\bibitem[Giga et~al.(2017)Giga, Kirshtein, and Liu]{Giga2017}
Mi-Ho Giga, Arkadz Kirshtein, and Chun Liu.
\newblock Variational modeling and complex fluids.
\newblock In Yoshikazu Giga and Antonin Novotny, editors, \emph{Handbook of
  Mathematical Analysis in Mechanics of Viscous Fluids}, pages 1--41. Springer
  International Publishing, 2017.

\bibitem[Gurtin et~al.(1984)Gurtin, MacCamy, and
  Socolovsky]{gurtin1984coordinate}
Morton~E Gurtin, Richard~C MacCamy, and Eduardo~A Socolovsky.
\newblock A coordinate transformation for the porous media equation that
  renders the free boundary stationary.
\newblock \emph{Quarterly of applied mathematics}, 42\penalty0 (3):\penalty0
  345--357, 1984.

\bibitem[Hoff(1985)]{hoff1985linearly}
David Hoff.
\newblock A linearly implicit finite-difference scheme for the one-dimensional
  porous medium equation.
\newblock \emph{Mathematics of computation}, 45\penalty0 (171):\penalty0
  23--33, 1985.

\bibitem[J{\"a}ger and Ka{\v{c}}ur(1991)]{jager1991solution}
Willi J{\"a}ger and Jozef Ka{\v{c}}ur.
\newblock Solution of porous medium type systems by linear approximation
  schemes.
\newblock \emph{Numerische Mathematik}, 60\penalty0 (1):\penalty0 407--427,
  1991.

\bibitem[Jin et~al.(1998)Jin, Pareschi, and Toscani]{jin1998diffusive}
Shi Jin, Lorenzo Pareschi, and Giuseppe Toscani.
\newblock Diffusive relaxation schemes for multiscale discrete-velocity kinetic
  equations.
\newblock \emph{SIAM Journal on Numerical Analysis}, 35\penalty0 (6):\penalty0
  2405--2439, 1998.

\bibitem[Jordan et~al.(1998)Jordan, Kinderlehrer, and
  Otto]{jordan1998variational}
Richard Jordan, David Kinderlehrer, and Felix Otto.
\newblock The variational formulation of the fokker--planck equation.
\newblock \emph{SIAM journal on mathematical analysis}, 29\penalty0
  (1):\penalty0 1--17, 1998.

\bibitem[Junge et~al.(2017)Junge, Matthes, and Osberger]{junge2017fully}
Oliver Junge, Daniel Matthes, and Horst Osberger.
\newblock A fully discrete variational scheme for solving nonlinear
  fokker--planck equations in multiple space dimensions.
\newblock \emph{SIAM Journal on Numerical Analysis}, 55\penalty0 (1):\penalty0
  419--443, 2017.

\bibitem[Kala{\v{s}}nikov(1967)]{kalavsnikov1967formation}
A~S Kala{\v{s}}nikov.
\newblock Formation of singularities in solutions of the equation of
  nonstationary filtration.
\newblock \emph{Z. Vycisl. Mat. i Mat. Fiz}, 7:\penalty0 440--444, 1967.

\bibitem[Keller and Segel(1970)]{keller1970initiation}
Evelyn~F Keller and Lee~A Segel.
\newblock Initiation of slime mold aggregation viewed as an instability.
\newblock \emph{Journal of theoretical biology}, 26\penalty0 (3):\penalty0
  399--415, 1970.

\bibitem[Knerr(1977)]{knerr1977porous}
Barry~F Knerr.
\newblock The porous medium equation in one dimension.
\newblock \emph{Transactions of the American Mathematical Society},
  234\penalty0 (2):\penalty0 381--415, 1977.

\bibitem[Lacey et~al.(1982)Lacey, Ockendon, and Tayler]{lacey1982waiting}
AA~Lacey, JR~Ockendon, and AB~Tayler.
\newblock “waiting-time” solutions of a nonlinear diffusion equation.
\newblock \emph{SIAM Journal on Applied Mathematics}, 42\penalty0 (6):\penalty0
  1252--1264, 1982.

\bibitem[Larsen and Pomraning(1980)]{larsen1980asymptotic}
E~W Larsen and G~C Pomraning.
\newblock Asymptotic analysis of nonlinear marshak waves.
\newblock \emph{SIAM Journal on Applied Mathematics}, 39\penalty0 (2):\penalty0
  201--212, 1980.

\bibitem[Larson and Bengzon(2013)]{larson2013finite}
Mats~G Larson and Fredrik Bengzon.
\newblock \emph{The finite element method: theory, implementation, and
  applications}, volume~10.
\newblock Springer Science \& Business Media, 2013.

\bibitem[Liu(2009)]{liu2009introduction}
Chun Liu.
\newblock An introduction of elastic complex fluids: an energetic variational
  approach.
\newblock In \emph{Multi-Scale Phenomena in Complex Fluids: Modeling, Analysis
  and Numerical Simulation}, pages 286--337. World Scientific, 2009.

\bibitem[Liu and Wu(2019)]{liu2019energetic}
Chun Liu and Hao Wu.
\newblock An energetic variational approach for the cahn--hilliard equation
  with dynamic boundary condition: model derivation and mathematical analysis.
\newblock \emph{Archive for Rational Mechanics and Analysis}, 233\penalty0
  (1):\penalty0 167--247, 2019.

\bibitem[Liu et~al.(2011)Liu, Shu, and Zhang]{liu2011high}
Yuanyuan Liu, Chi-Wang Shu, and Mengping Zhang.
\newblock High order finite difference weno schemes for nonlinear degenerate
  parabolic equations.
\newblock \emph{SIAM Journal on Scientific Computing}, 33\penalty0
  (2):\penalty0 939--965, 2011.

\bibitem[Maire et~al.(2007)Maire, Abgrall, Breil, and Ovadia]{maire2007cell}
Pierre-Henri Maire, R{\'e}mi Abgrall, J{\'e}r{\^o}me Breil, and Jean Ovadia.
\newblock A cell-centered lagrangian scheme for two-dimensional compressible
  flow problems.
\newblock \emph{SIAM Journal on Scientific Computing}, 29\penalty0
  (4):\penalty0 1781--1824, 2007.

\bibitem[Matthes and Osberger(2017)]{matthes2017convergent}
Daniel Matthes and Horst Osberger.
\newblock A convergent lagrangian discretization for a nonlinear fourth-order
  equation.
\newblock \emph{Foundations of Computational Mathematics}, 17\penalty0
  (1):\penalty0 73--126, 2017.

\bibitem[Monsaingeon(2016)]{monsaingeon2016explicit}
L{\'e}onard Monsaingeon.
\newblock An explicit finite-difference scheme for one-dimensional generalized
  porous medium equations: Interface tracking and the hole filling problem.
\newblock \emph{ESAIM: Mathematical Modelling and Numerical Analysis},
  50\penalty0 (4):\penalty0 1011--1033, 2016.

\bibitem[Nakaki and Tomoeda(2003)]{nakaki2003numerical}
Tatsuyuki Nakaki and Kenji Tomoeda.
\newblock Numerical approach to the waiting time for the one-dimensional porous
  medium equation.
\newblock \emph{Quarterly of Applied Mathematics}, 61\penalty0 (4):\penalty0
  601--612, 2003.

\bibitem[Ngo and Huang(2017)]{ngo2017study}
Cuong Ngo and Weizhang Huang.
\newblock A study on moving mesh finite element solution of the porous medium
  equation.
\newblock \emph{Journal of Computational Physics}, 331:\penalty0 357--380,
  2017.

\bibitem[Ngo and Huang(2018)]{ngo2018adaptive}
Cuong Ngo and Weizhang Huang.
\newblock Adaptive finite element solution of the porous medium equation in
  pressure formulation.
\newblock \emph{Numerical Methods for Partial Differential Equations}, 2018.

\bibitem[Oleinik et~al.(1958)Oleinik, Kalashnikov, and
  Juj-lin]{oleinik1958cauchy}
Olga~Arsen'evna Oleinik, Anatolii~Sergeevich Kalashnikov, and Czou Juj-lin.
\newblock The cauchy problem and boundary problems for equations of the type of
  non-stationary filtration.
\newblock \emph{Izvestiya Rossiiskoi Akademii Nauk. Seriya Matematicheskaya},
  22\penalty0 (5):\penalty0 667--704, 1958.

\bibitem[Onsager(1931{\natexlab{a}})]{onsager1931reciprocal}
Lars Onsager.
\newblock Reciprocal relations in irreversible processes. i.
\newblock \emph{Physical review}, 37\penalty0 (4):\penalty0 405,
  1931{\natexlab{a}}.

\bibitem[Onsager(1931{\natexlab{b}})]{onsager1931reciprocal2}
Lars Onsager.
\newblock Reciprocal relations in irreversible processes. ii.
\newblock \emph{Physical review}, 38\penalty0 (12):\penalty0 2265,
  1931{\natexlab{b}}.

\bibitem[Otto(2001)]{Otto2001PME}
Felix Otto.
\newblock The geometry of dissipative evolution equations: The porous medium
  equation.
\newblock \emph{Communications in Partial Differential Equations}, 26\penalty0
  (1-2):\penalty0 101--174, 2001.
\newblock \doi{10.1081/PDE-100002243}.

\bibitem[Pattle(1959)]{pattle1959diffusion}
E~Pattle, R.
\newblock Diffusion from an instantaneous point source with a
  concentration-dependent coefficient.
\newblock \emph{The Quarterly Journal of Mechanics and Applied Mathematics},
  12\penalty0 (4):\penalty0 407--409, 1959.

\bibitem[Persson and Strang(2004)]{persson2004simple}
Per-Olof Persson and Gilbert Strang.
\newblock A simple mesh generator in matlab.
\newblock \emph{SIAM review}, 46\penalty0 (2):\penalty0 329--345, 2004.

\bibitem[Perthame et~al.(2014)Perthame, Quir{\'o}s, and
  V{\'a}zquez]{perthame2014hele}
Beno{\^\i}t Perthame, Fernando Quir{\'o}s, and Juan~Luis V{\'a}zquez.
\newblock The hele--shaw asymptotics for mechanical models of tumor growth.
\newblock \emph{Archive for Rational Mechanics and Analysis}, 212\penalty0
  (1):\penalty0 93--127, 2014.

\bibitem[Socolovsky(1988{\natexlab{a}})]{socolovsky1988lagrangian}
E~A Socolovsky.
\newblock Lagrangian non-oscillatory and fem schemes for the porous media
  equation.
\newblock \emph{Computers \& Mathematics with Applications}, 15\penalty0
  (6-8):\penalty0 611--617, 1988{\natexlab{a}}.

\bibitem[Socolovsky(1988{\natexlab{b}})]{socolovsky1988numerical}
E~A Socolovsky.
\newblock On the numerical approximation of finite speed diffusion problems.
\newblock \emph{Numerische Mathematik}, 53\penalty0 (1-2):\penalty0 97--105,
  1988{\natexlab{b}}.

\bibitem[Strutt(1871)]{strutt1871some}
W~Strutt, J.
\newblock Some general theorems relating to vibrations.
\newblock \emph{Proceedings of the London Mathematical Society}, 1\penalty0
  (1):\penalty0 357--368, 1871.

\bibitem[Sun and Liu(2009)]{sun2009energetic}
Huan Sun and Chun Liu.
\newblock On energetic variational approaches in modeling the nematic liquid
  crystal flows.
\newblock \emph{Discrete and Continuous Dynamical Systems}, 23\penalty0
  (1-2):\penalty0 455--475, 2009.

\bibitem[Tomoeda et~al.(1983)Tomoeda, Mimura, et~al.]{tomoeda1983numerical}
Kenji Tomoeda, Masayasu Mimura, et~al.
\newblock Numerical approximations to interface curves for a porous media
  equation.
\newblock \emph{Hiroshima mathematical journal}, 13\penalty0 (2):\penalty0
  273--294, 1983.

\bibitem[Topaz et~al.(2006)Topaz, Bertozzi, and Lewis]{topaz2006nonlocal}
Chad~M Topaz, Andrea~L Bertozzi, and Mark~A Lewis.
\newblock A nonlocal continuum model for biological aggregation.
\newblock \emph{Bulletin of mathematical biology}, 68\penalty0 (7):\penalty0
  1601, 2006.

\bibitem[V{\'a}zquez(2007)]{vazquez2007porous}
Juan~Luis V{\'a}zquez.
\newblock \emph{The porous medium equation: mathematical theory}.
\newblock Oxford University Press, 2007.

\bibitem[Wang et~al.(2017)Wang, Zhang, and Chen]{wang2017topological}
Yiwei Wang, Pingwen Zhang, and Jeff Z.~Y Chen.
\newblock Topological defects in an unconfined nematic fluid induced by single
  and double spherical colloidal particles.
\newblock \emph{Physical Review E}, 96\penalty0 (4):\penalty0 042702, 2017.

\bibitem[Wang et~al.(2018)Wang, Zhang, and Chen]{wang2018formation}
Yiwei Wang, Pingwen Zhang, and Jeff~ZY Chen.
\newblock Formation of three-dimensional colloidal crystals in a nematic liquid
  crystal.
\newblock \emph{Soft matter}, 14\penalty0 (32):\penalty0 6756--6766, 2018.

\bibitem[Westdickenberg and Wilkening(2010)]{westdickenberg2010variational}
Michael Westdickenberg and Jon Wilkening.
\newblock Variational particle schemes for the porous medium equation and for
  the system of isentropic euler equations.
\newblock \emph{ESAIM: Mathematical Modelling and Numerical Analysis},
  44\penalty0 (1):\penalty0 133--166, 2010.

\bibitem[Wilhelm~Alt and Luckhaus(1983)]{wilhelm1983quasilinear}
Hans Wilhelm~Alt and Stephan Luckhaus.
\newblock Quasilinear elliptic-parabolic differential equations.
\newblock \emph{Mathematische Zeitschrift}, 183\penalty0 (3):\penalty0
  311--341, 1983.

\bibitem[Witelski(1997)]{witelski1997segregation}
Thomas~P Witelski.
\newblock Segregation and mixing in degenerate diffusion in population
  dynamics.
\newblock \emph{Journal of Mathematical Biology}, 35\penalty0 (6):\penalty0
  695--712, 1997.

\bibitem[Xu et~al.(2014)Xu, Sheng, and Liu]{xu2014energetic}
Shixin Xu, Ping Sheng, and Chun Liu.
\newblock An energetic variational approach for ion transport.
\newblock \emph{Communications in Mathematical Sciences}, 12\penalty0
  (4):\penalty0 779--789, 2014.

\bibitem[Xu et~al.(2016)Xu, Di, and Doi]{xu2016variational}
Xianmin Xu, Yana Di, and Masao Doi.
\newblock Variational method for liquids moving on a substrate.
\newblock \emph{Physics of Fluids}, 28\penalty0 (8):\penalty0 087101, 2016.

\bibitem[Zhang and Wu(2009)]{zhang2009numerical}
Qiang Zhang and Zi-Long Wu.
\newblock Numerical simulation for porous medium equation by local
  discontinuous galerkin finite element method.
\newblock \emph{Journal of Scientific Computing}, 38\penalty0 (2):\penalty0
  127--148, 2009.

\end{thebibliography}

\end{document}